\documentclass[12pt]{amsart}
\usepackage{SSdefn}

\DeclareMathOperator{\Ob}{Ob}

\DeclareMathOperator{\ch}{ch}

\newcommand{\alg}{\mathrm{alg}} 
\newcommand{\gl}{\mathfrak{gl}}
\newcommand{\gen}{\mathrm{gen}}

\newcommand{\Cl}{\mathrm{Cl}}

\newcommand{\SVec}{\mathbf{SVec}}
\newcommand{\QVec}{\mathbf{QVec}}
\newcommand{\QPol}{\mathbf{QPol}}

\author{Rohit Nagpal}
\address{Department of Mathematics, University of Michigan, Ann Arbor, MI}
\email{\href{mailto:rohitna@umich.edu}{rohitna@umich.edu}}
\urladdr{\url{http://www-personal.umich.edu/~rohitna/}}

\author{Steven V Sam}
\address{Department of Mathematics, University of California, San Diego, CA}
\email{\href{mailto:ssam@ucsd.edu}{ssam@ucsd.edu}}
\urladdr{\url{http://math.ucsd.edu/~ssam/}}

\author{Andrew Snowden}
\address{Department of Mathematics, University of Michigan, Ann Arbor, MI}
\email{\href{mailto:asnowden@umich.edu}{asnowden@umich.edu}}
\urladdr{\url{http://www-personal.umich.edu/~asnowden/}}
\thanks{RN was partially supported by NSF DMS-1638352. SS was partially supported by NSF grant DMS-1812462. AS was supported by NSF grant DMS-1453893.}

\subjclass[2010]{%
13E05, %Noetherian rings and modules
13A50.%Actions of groups on commutative rings; invariant theory
}

\date{August 18, 2021}

\title[On the geometry and representation theory of isomeric matrices]{On the geometry and representation theory\\ of isomeric matrices}

\begin{document}

\begin{abstract}
The space of $n \times m$ complex matrices can be regarded as an algebraic variety on which the group $\GL_n \times \GL_m$ acts. There is a rich interaction between geometry and representation theory in this example. In an important paper, de Concini, Eisenbud, and Procesi classified the equivariant ideals in the coordinate ring. More recently, we proved a noetherian result for families of equivariant modules as $n$ and $m$ vary. In this paper, we establish analogs of these results for the space of $(n|n) \times (m|m)$ isomeric matrices with respect to the action of $\bQ_n \times \bQ_m$, where $\bQ_n$ is the automorphism group of the isomeric structure (commonly known as the ``queer supergroup''). Our work is motivated by connections to the Brauer category and the theory of twisted commutative algebras.
\end{abstract}

\maketitle
\tableofcontents

\section{Introduction}

\subsection{Background}

Let $V$ and $W$ be finite dimensional complex vector spaces, and consider the space $\Hom(V,W)$ of all linear maps $V \to W$, regarded as an affine algebraic variety. The group $\GL(V) \times \GL(W)$ acts on this variety, and the decomposition of its coordinate ring is well-known:
\begin{displaymath}
\bC[\Hom(V,W)] = \Sym(V \otimes W^*) = \bigoplus_{\lambda} \bS_{\lambda}(V) \otimes \bS_{\lambda}(W^*)
\end{displaymath}
Here the sum is over all partitions $\lambda$ and $\bS_{\lambda}$ denotes the Schur functor associated to $\lambda$. An important feature of this decomposition is that it is multiplicity free: each non-zero summand is irreducible, and no two such summands are isomorphic.

The interaction between the geometry of the variety $\Hom(V,W)$ and the representation theory of its coordinate ring is a rich topic of study. An important theorem in this direction is the classification of equivariant ideals, established in \cite{CEP}: the key statement is that the ideal generated by the $\lambda$ summand is exactly the sum of the $\mu$ summands over all $\mu$ with $\lambda \subset \mu$. (Here $\subset$ denotes containment of Young diagrams.) Another theorem along these lines, that we proved \cite{sym2noeth}, concerns finiteness properties of equivariant modules that are uniform in $V$ and $W$: one can regard $(V,W) \mapsto \Sym(V \otimes W)$ as an algebra object in the category of bivariate polynomial functors, and we show that it is noetherian. (Omitting duality does not change much, and makes the functorial properties nicer.) The proof of this theorem crucially relies upon the aforementioned result of \cite{CEP}.

Similar pictures are known in some related contexts. For example, consider the space $\Sym^2(V^*)$ of symmetric bilinear forms on $V$. The coordinate ring $\Sym(\Sym^2(V))$ is again multiplicity free as a representation of $\GL(V)$. The equivariant ideals were classified in \cite{abeasis}, and the noetherian result was established in \cite{sym2noeth}. See \cite{pfaffians,periplectic,symsp1} for some other cases.

The purpose of this paper is to establish analogs of the above results in a new situation. An {\bf isomeric vector space} is a super vector space $V$ equipped with an odd-degree isomorphism $\alpha \colon V \to V$ squaring to the identity. The {\bf isomeric supergroup} (also known as the {\bf queer supergroup}) $\bQ(V)$ is the automorphism supergroup of $(V, \alpha)$. Given two isomeric vector spaces $(V,\alpha)$ and $(W,\beta)$, one can consider the supervariety $X$ of all linear maps $V \to W$ that are compatible with the isomeric structure. The supergroup $\bQ(V) \times \bQ(W)$ acts on $X$. We classify the equivariant ideals of $\bC[X]$ and establish a noetherian result.

\subsection{Statement of results}

We now state our main results in detail. Let $\fq$ be the infinite isomeric Lie superalgebra (see \S \ref{ss:isomeric}). There is a notion of polynomial representation for $\fq$, analogous to that for the general linear group (see \S \ref{ss:polrep}). A {\bf (bivariate) isomeric algebra} is a supercommutative superalgebra equipped with an action of $\fq \times \fq$ under which it forms a polynomial representation. Let $\bV$ be the standard representation of $\fq$ and let $\bU$ be the half tensor product $2^{-1} (\bV \otimes \bV)$, which is a polynomial representation of $\fq \times \fq$. (See \S \ref{ss:qvec} for the definition of half tensor product.) Let $A=\Sym(\bU)$. This is an isomeric algebra, and the main one of interest in this paper. The space $\Spec(A)$ is (more or less) the space of infinite isomeric matrices.

The irreducible polynomial representations of $\fq$ are parametrized by strict partitions. Let $T_{\lambda}$ be the irreducible representation corresponding to the strict partition $\lambda$. An analog of the Cauchy decomposition yields the decomposition
\begin{equation} \label{eq:intro-cauchy}
A = \bigoplus 2^{-\delta(\lambda)}(T_{\lambda} \otimes T_{\lambda}),
\end{equation}
where the sum is over all strict partitions $\lambda$, $\delta(\lambda)$ is either~0 or~1, and a $2^{-1}$ indicates a half tensor product. Each summand here is an irreducible representation of $\fq \times \fq$. This shows that $A$ is multiplicity free as a representation of $\fq \times \fq$. In particular, we see that an equivariant ideal of $A$ (meaning one stable by $\fq \times \fq$) is determined by which summands it contains. Our first main result is the following, which completely classifies the equivariant ideals of $A$:

\begin{theorem} \label{mainthm1}
Let $I^{\lambda}$ be the ideal of $A$ generated by the $\lambda$ summand in \eqref{eq:intro-cauchy}. Then $I^{\lambda}$ contains the $\mu$ summand if and only if $\lambda \subseteq \mu$.
\end{theorem}

The key idea in proving Theorem~\ref{mainthm1} is to use Schur--Sergeev duality to convert to a more combinatorial problem; see \S \ref{ss:class-overview} for a more detailed outline. We note that this theorem is the isomeric analog of the aforesaid result of \cite{CEP}. As a simple application of this theorem, we describe the isomeric determinantal ideals (see Remark~\ref{rmk:detideal}).

Theorem~\ref{mainthm1} implies that the lattice of ideals of $A$ is isomorphic to the lattice of ideals in the poset of strict partitions. From this, it follows that equivariant ideals of $A$ satisfy the ascending chain condition. This is a noetherian-like result; however, in practice one wants a much stronger statement.

Let $\Mod_A$ be the category of $\fq \times \fq$ equivariant $A$-modules that form a polynomial representation of $\fq \times \fq$. In other words, treating $A$ as an algebra object of the tensor category $\Rep^{\pol}(\fq \times \fq)$, the category $\Mod_A$ is the category of $A$-module objects. We say that $A$ is {\bf noetherian} (as an isomeric algebra) if the category $\Mod_A$ is locally noetherian. Equivalently, this means that any subobject of a finitely generated object of $\Mod_A$ is again finitely generated. The corollary of Theorem~\ref{mainthm1} observed above shows that any subobject of $A$ itself is finitely generated. This is not enough to conclude that $A$ is noetherian: a general finitely generated $A$-module is a quotient of $V \otimes A$ for some finite length polynomial representation $V$, but typically not a quotient of a finite direct sum of $A$'s. Our second main theorem yields the desired strengthening:

\begin{theorem} \label{mainthm2}
The bivariate isomeric algebra $A$ is noetherian.
\end{theorem}

The proof of Theorem~\ref{mainthm2} follows the general plan employed in previous noetherian results we have proved \cite{sym2noeth, periplectic, symsp1}. The basic idea is to break $\Mod_A$ up into two pieces: the torsion subcategory $\Mod_A^{\tors}$ and the generic category $\Mod_A^{\gen}$ (which is the Serre quotient by the torsion subcategory). Thanks to Theorem~\ref{mainthm1}, one can show that $\Mod_A^{\tors}$ is locally noetherian. The crucial step is to identify $\Mod_A^{\gen}$ with the category of algebraic representations of $\fq$ as studied in \cite{serganova}. We thus see that $\Mod_A^{\gen}$ inherits the pleasant properties of this representation category (such as: all objects are locally finite length, and finite length objects have finite injective dimension). We then piece together the information about $\Mod_A^{\tors}$ and $\Mod_A^{\gen}$ to obtain our result on $\Mod_A$.

We note that our proof of Theorem~\ref{mainthm2} yields quite a bit of extra useful information. For example, as stated above, we determine the structure of the generic category $\Mod_A^{\gen}$. We also prove some important results about the section functor $S \colon \Mod_A^{\gen} \to \Mod_A$. Some further results can be easily deduced: for instance, projective $A$-modules are injective.

\begin{remark}
The simplest example of an isomeric algebra is $\Sym(\bV)$. It is easily seen to be noetherian, see Remark~\ref{rmk:single-var}. The algebra $A$ is, in a sense, the smallest isomeric algebra where the noetherian property is not obvious.
\end{remark}

\subsection{Motivation}

In \cite{brauercat1}, the second and third authors investigate the Brauer category and its relatives (this is actually the first in a series of papers on this topic). One relative is the ``isomeric walled Brauer category,'' which relates to mixed tensor representations of the isomeric algebra. Theorem~\ref{mainthm2} implies a noetherian result for this category that will be important for that project. This was our motivation for studying $A$ specifically.

We also have a more general motivation for the current project. A {\bf $\GL$-algebra} is a commutative algebra equipped with an action of the infinite general linear group $\GL$ under which it forms a polynomial representation. In characteristic~0, these are equivalent to twisted commutative algebras (tca's) by Schur--Weyl duality. In recent years, $\GL$-algebras and tca's have received a lot of attention. All indications so far are that these are very well-behaved objects, though there is much that is still unknown. The isomeric algebras studied in this paper are close relatives of $\GL$-algebras. Given that $\GL$-algebras are such a rich topic, we expect that isomeric algebras are as well. This paper represents the first attempt to study them in any serious capacity.

\subsection{Future work}

As mentioned, it seems that isomeric algebras are likely just as rich as $\GL$-algebras. Developing the theory of isomeric algebras in more detail is therefore a potentially fruitful direction for future work.

A more specific problem is to classify the equivariant prime ideals of the isomeric algebra $A$ studied in this paper. The third author tackled the analogous problem for the $\GL$-algebra $\Sym(\Sym^2(\bC^{\infty}))$ in \cite{tcaspec}. He is currently pursuing this problem with Robert Laudone, and the result will appear in the forthcoming work \cite{LS}.

\subsection{A remark on terminology}

The supergroup $\bQ(n)$ has traditionally been called the ``queer supergroup.'' Since the connotations of this word have shifted over the years, we have introduced the term ``isomeric'' to take its place\footnote{The word ``isomeric'' is not new: it occurs in chemistry as the adjectival form of ``isomer.'' However, it does not appear to be used within pure mathematics.}. This word is derived from the Greek words \emph{isos}, meaning equal, and \emph{m\'eros}, meaning part, and is intended to reflect the fact that an isomeric structure on a super vector space makes its two parts isomorphic. We hope that other authors will choose to use this new terminology as well. We thank the conscientious editor who suggested that we change the terminology.

\subsection{Outline}

In \S \ref{s:isomeric}, we review isomeric vector spaces, the isomeric Lie superalgebra, and isomeric algebras. In \S \ref{s:class}, we classify the ideals of $A$. In \S \ref{s:local}, we show that equivariant $A$-modules are locally free at a generic point of $\Spec(A)$. Using this, we describe the generic category $\Mod_A^{\gen}$ in \S \ref{s:gen}. Finally, we prove the noetherian result in \S \ref{s:noeth}.

\subsection{Notation}

We list the most important notation:

\begin{description}[align=right,labelwidth=2.5cm,leftmargin=!]
\item [$\zeta$] a fixed square root of $-1$ in $\bC$
\item [$\bV$] the complex super vector space with basis $\{e_i,f_i\}_{i \ge 1}$
\item [$\alpha$] the isomeric structure on $\bV$ given by $\alpha(e_i)=f_i$
\item [$\fq$] the isomeric algebra associated to $(\bV, \alpha)$
\item [$2^{-1}(-\otimes -)$] the half tensor product (see \S \ref{ss:qvec})
\item [$\bU$] the half-tensor product $2^{-1}(\bV \otimes \bV)$, a representation of $\fq \times \fq$
\item [$A$] the $\fq \times \fq$-algebra $\Sym(\bU)$
\item [$\bT_{\lambda}$] the isomeric Schur functor associated to the strict partition $\lambda$
\end{description}

\section{Isomeric algebras} \label{s:isomeric}

\subsection{Super vector spaces}

A {\bf super vector space} is a $\bZ/2$-graded complex vector space $V=V_0 \oplus V_1$. For a homogeneous element $x \in V$, we write $\vert x \vert \in \bZ/2$ for its degree. We let $\bC^{n|m}$ be the super vector space with degree~0 part $\bC^n$ and degree~1 part $\bC^m$. We let $e_1, \ldots, e_n, f_1, \ldots, f_m$ be the standard basis for $\bC^{n|m}$, where the $e_i$'s have degree~0 and the $f_i$'s have degree~1. We let $\bV=\bC^{\infty|\infty}=\bigcup_{n \ge 1} \bC^{n|n}$, and sometimes write $\bW$ for a copy of $\bV$. Given a super vector space $V$ and $k \in \bZ/2$, we let $V[k]$ be the super vector space defined by $V[k]_i=V_{k+i}$. (We remark that $V[1]$ is often denote $\Pi(V)$ in the literature, e.g., in \cite{chengwang}.)

Let $V$ and $W$ be super vector spaces. We let $\Hom(V,W)$ denote the set of all linear maps $V \to W$. We let $\Hom(V,W)_k$ denote the subset of $\Hom(V,W)$ consisting of maps $f$ such that $f(V_i) \subset W_{i+k}$ for all $i$. Then $\Hom(V,W)=\Hom(V,W)_0 \oplus \Hom(V,W)_1$, and so $\Hom(V,W)$ is naturally a super vector space.

We let $\SVec$ denote the category whose objects are super vector spaces and whose maps are all linear maps. We let $\SVec^{\circ}$ be the subcategory where the morphisms are homogeneous of even degree.

Given super vector spaces $V$ and $W$, let $V \otimes W$ be the usual tensor product of $V$ and $W$, endowed with the usual grading. Given homogeneous linear maps $f \colon V \to V'$ and $g \colon W \to W'$, we let $f \otimes g$ be the linear map $V \otimes W \to V' \otimes W'$ defined on homogeneous elements by
\begin{equation} \label{eq:sign-rule}
(f \otimes g)(v \otimes w) = (-1)^{|g||v|} f(v) \otimes g(w).
\end{equation}
This construction endows $\SVec^{\circ}$ with a monoidal structure. This monoidal structure admits a symmetry $\tau$ defined by $\tau(x \otimes y)=(-1)^{\vert x \vert \vert y \vert} y \otimes x$. We note that $\otimes$ does not define a monoidal structure on $\SVec$ in the usual sense, as it does not even define a functor $\SVec \times \SVec \to \SVec$ due to \eqref{eq:sign-rule}; however, it does define a monoidal structure on $\SVec$ in the sense of supercategories, as we now discuss.

\subsection{Supercategories}

A {\bf supercategory} is a category enriched in super vector spaces, or, more precisely, the category $\SVec^{\circ}$. Concretely, a supercategory is just a linear category in which there is a notion of even and odd homogeneous morphism, satisfying some rules; most categories built out of super vector spaces will therefore be supercategories. We refer to \cite[\S 2]{supermon} as a general reference, though this concept has been discussed elsewhere as well (e.g., \cite{brundan}). Given a supercategory $\cC$, we let $\cC^{\circ}$ be the ordinary category with the same objects and with $\Hom_{\cC^{\circ}}(X,Y)=\Hom_{\cC}(X,Y)_0$.

There is a natural notion of product of supercategories, which leads to the notion of a monoidal supercategory; see \cite[\S 3.1]{supermon} or \cite[Definition~1.4]{brundan}. If $(\cC, \otimes)$ is a monoidal supercategory then $\otimes$ induces a monoidal operation on $\cC^{\circ}$. The monoidal product $\otimes$ is defined on all pairs of morphisms, and satisfies the sign rule \eqref{eq:sign-rule}. The notion of a symmetry on a monoidal supercategory is defined in the usual manner. (We note that the paper \cite{supermon} is devoted to studying the notion of a \emph{supersymmetry} on a monoidal supercategory, which is more complicated than the notion of a symmetry. We will not need supersymmetries in this paper.) The basic example of a monoidal supercategory is $\SVec$ with the monoidal operation $\otimes$ discussed above; $\tau$ defines a symmetry on this monoidal structure.

\subsection{Isomeric vector spaces} \label{ss:qvec}

A {\bf isomeric vector space} is a pair $(V,\alpha)$ consisting of a super vector space $V$ together with an odd endomorphism $\alpha \colon V \to V$ such that $\alpha^2 = 1$. We endow $\bC^{n|n}$ with an isomeric structure by $\alpha(e_i)=f_i$ and $\alpha(f_i)=e_i$; we call this the standard isomeric structure. We similarly define the standard isomeric structure on $\bV=\bC^{\infty|\infty}$. Every finite dimensional isomeric vector space is isomorphic to some $\bC^{n|n}$ equipped with the standard isomeric structure. A homogeneous morphism $f \colon (V,\alpha) \to (W,\beta)$ of isomeric vector spaces is a homogeneous linear map $f \colon V \to W$ such that $f \alpha = (-1)^{\vert f \vert} \beta f$; a general morphism is a sum of homogeneous morphisms. We let $\QVec$ denote the supercategory of isomeric vector spaces.

Let $(V,\alpha)$ and $(W,\beta)$ be two isomeric vector spaces. Then $\alpha \otimes \beta$ is an even endomorphism of $V \otimes W$ that squares to $-1$; the sign comes from \eqref{eq:sign-rule}. Fix once and for all a square root of $-1$ in $\bC$, denoted $\zeta$. We define the {\bf half tensor product} of $(V,\alpha)$ and $(W,\beta)$, denoted $2^{-1} (V \otimes W)$, to be the $\zeta$ eigenspace of $\alpha \otimes \beta$. We note that $\alpha \otimes 1$ induces an isomorphism between the $\zeta$ and $-\zeta$ eigenspaces, and so we have a natural isomorphism $V \otimes W \cong 2^{-1}(V \otimes W) \otimes \bC^{1|1}$. Abstractly, the half tensor product can be understood using the theory of supersymmetric monoidal categories developed in \cite{supermon}, but this perspective will not be needed in this paper.

\subsection{The isomeric Lie superalgebra} \label{ss:isomeric}

Let $(V,\alpha)$ be a finite dimensional isomeric vector space. The {\bf isomeric Lie superalgebra} $\fq(V)=\fq(V,\alpha)$ is the set of endomorphisms of $(V,\alpha)$ in the category $\QVec$. It forms a sub Lie superalgebra of the general linear Lie superalgebra $\gl(V)$. We let $\fq_n=\fq(\bC^{n|n})$. Explicitly, an element of $\fq_n$ can be described as a matrix of the form
\begin{displaymath}
\begin{pmatrix} a & b \\ -b & a \end{pmatrix}
\end{displaymath}
where $a$ and $b$ are $n \times n$ matrices. We define the {\bf Chevalley automorphism} $\tau$ of $\fq_n$ by
\begin{displaymath}
\tau \begin{pmatrix} a & b \\ -b & a \end{pmatrix}
= -\begin{pmatrix} a^t & \zeta b^t \\ -\zeta b^t & a^t \end{pmatrix}
\end{displaymath}
where $(-)^t$ denotes the usual matrix transpose. One verifies that this is a Lie superalgebra homomorphism. Note that
\begin{displaymath}
\tau^2 \begin{pmatrix} a & b \\ -b & a \end{pmatrix}
= \begin{pmatrix} a & -b \\ b & a \end{pmatrix},
\end{displaymath}
and so $\tau$ has order four. We let $\fq=\fq_{\infty}=\bigcup_{n \ge 1} \fq_n$ be the infinite isomeric superalgebra. It can be represented by matrices as above, and we define $\tau$ on $\fq$ as above.

\subsection{Polynomial representations} \label{ss:polrep}

The space $\bC^{n|n}$ is naturally a representation of $\fq_n$, and called the standard representation. We say that a representation of $\fq_n$ is {\bf polynomial} if it occurs as a subquotient of a (possibly infinite) direct sum of tensor powers of the standard representation. We let $\Rep^{\pol}(\fq_n)$ denote the supercategory of such representations. It is easily seen to be closed under tensor product. It follows from Schur--Sergeev duality \cite[\S 3.4]{chengwang} that the abelian category $\Rep^{\pol}(\fq_n)^{\circ}$ is semi-simple. Recall that a {\bf strict partition} of $n$ is a partition $\lambda$ of $n$ such that $\lambda_1>\lambda_2>\cdots>\lambda_r>0$; with this notation, we write $\vert \lambda \vert=n$ and $\ell(\lambda)=r$. For each strict partition $\lambda$ with $\ell(\lambda) \le n$, there is a simple object $T_{\lambda,n}$ of $\Rep^{\pol}(\fq_n)^{\circ}$, and every simple object is isomorphic to some $T_{\lambda,n}$ or $T_{\lambda,n}[1]$. To be a bit more precise, define
\begin{displaymath}
\delta(\lambda) = \begin{cases} 0 & \text{if $\ell(\lambda)$ is even} \\
1 & \text{if $\ell(\lambda)$ is odd} \end{cases}.
\end{displaymath}
Then there is an even isomorphism $T_{\lambda,n} \cong T_{\lambda,n}[1]$ if $\delta(\lambda)=1$, and these are the only isomorphisms among objects of the form $T_{\mu,n}[k]$ with $k \in \bZ/2$.

The above discussion applies equally well to the infinite isomeric algebra $\fq$. Precisely, we define a representation of $\fq$ to be polynomial if it occurs as a subquotient of a direct sum of tensor powers of the standard representation $\bV$. The category $\Rep^{\pol}(\fq)^{\circ}$ is again semi-simple and closed under tensor products. Given a strict partition $\lambda$, there is a simple object $T_{\lambda}$, and every simple object is isomorphic to some $T_{\lambda}$ or $T_{\lambda}[1]$. We define the notion of polynomial representation of $\fq \times \fq$ in an analogous manner. Similar results hold for it (e.g., its simple objects have the form $T_{\lambda} \otimes T_{\mu}$ or $T_{\lambda} \otimes T_{\mu} [1]$ with $\lambda$ and $\mu$ strict partitions).

\subsection{Polynomial functors}

Consider the supercategory $\Fun(\QVec,\SVec)$ of all functors $\QVec \to \SVec$. The subcategory $\Fun(\QVec,\SVec)^{\circ}$ consisting of even degree morphisms is abelian. Define a functor $T_n \colon \QVec \to \SVec$ by $T_n(V,\alpha) = V^{\otimes n}$. By \cite[Theorem 3.49]{chengwang}, this functor is semisimple, and decomposes in $\Fun(\QVec,\SVec)^{\circ}$ as a direct sum of simple functors $\bT_\lambda$ (and their shifts) indexed by strict partitions $\lambda$ of $n$. We say that a functor $\QVec \to \SVec$ is {\bf polynomial} if it is a subquotient of a (possibly infinite) direct sum of the functors $T_n$ in the category $\Fun(\QVec,\SVec)^{\circ}$. We denote by $\QPol$ the full subcategory of $\Fun(\QVec,\SVec)$ spanned by polynomial functors. The category $\QPol^{\circ}$ is a semisimple abelian category, and every simple object is isomorphic to some $\bT_{\lambda}$ or $\bT_{\lambda}[1]$. The tensor product of two polynomial functors is again a polynomial functor.

We have an evaluation functor $\QPol \to \Rep^{\pol}(\fq_n)$ given by $F \mapsto F(\bC^{n|n})$. For a strict partition $\lambda$, we have
\begin{displaymath}
\bT_{\lambda}(\bC^{n|n}) = \begin{cases} T_{\lambda,n} & \text{if $n \ge \ell(\lambda)$} \\
0 & \text{otherwise} \end{cases}.
\end{displaymath}
We similarly have an evaluation functor $\QPol \to \Rep^{\pol}(\fq)$ that takes $\bT_{\lambda}$ to $T_{\lambda}$. This is an equivalence of supercategories.

We also have a notion of bivariate polynomial functors. Precisely, these are subquotients of direct sums of functors $T_{m,n} \colon \QVec \times \QVec \to \SVec$ given by $T_{m,n}(V,W) = V^{\otimes m} \otimes W^{\otimes n}$.We let $\QPol^{(2)}$ denote the category of such functors. It has similar properties to $\QPol$.

\subsection{Isomeric algebras}

An {\bf isomeric algebra} is a commutative algebra object in the symmetric monoidal category $\QPol^{\circ} \cong \Rep^{\pol}(\fq)^{\circ}$. A module for such an algebra is a module object in the usual sense. If $M$ is a module over the isomeric algebra $B$ then $M(V)$ is a $B(V)$-module for all isomeric spaces $V$. In this article, we are mostly interested in the bivariate analogue, i.e., commutative algebras in $\QPol^{(2)}$; we will also call them isomeric algebras and often omit the adjective ``bivariate.''

Let $M$ be a polynomial functor. We define $\ell(M)$ to be the supremum of the $\ell(\lambda)$ over $\lambda$ for which $\bT_{\lambda}$ occurs as a constituent of $M$. We make a similar definition for polynomial representations. We say that $M$ is {\bf bounded} if $\ell(M)<\infty$. The simple functor $\bT_\nu$ appears in $\bT_\lambda \otimes \bT_\mu$ if and only if $P_\nu$ has nonzero coefficient when the product $P_\lambda P_\mu$ is expanded in the basis of Schur P-functions (the Schur Q-functions are given by $Q_\lambda = 2^{\ell(\lambda)} P_\lambda$ \cite[III, (8.7)]{macdonald} and these are, up to powers of 2, the characters of $\bT_\lambda$ \cite[Theorem 3.51]{chengwang}). It follows from \cite[III, (8.18)]{macdonald} that the tensor product of two bounded representations is again bounded. We extend the notion of bounded to bivariate polynomial functors in the obvious way.

An isomeric algebra $A$ is finitely generated if it is a quotient of $\Sym(V)$ for some finite length object $V$ of $\QPol$ (or $\QPol^{(2)}$), and an $A$-module is finitely generated if it is a quotient of $A \otimes W$ for some finite length object $W$. We say that $A$ is {\bf noetherian} if every finitely generated $A$-module is noetherian, i.e., its submodules satisfy the ascending chain condition.

\begin{proposition} \label{prop:bounded-noeth}
Every finitely generated bounded isomeric algebra is noetherian.
\end{proposition}

\begin{proof}
If $A$ is bounded, then so is $A \otimes W$ for any finite length object $W$, and hence so is any finitely generated $A$-module $M$. If $\ell(M) \le r$, then given submodules $N \subset N' \subset M$, we have $N = N'$ if and only if $N(\bC^{r|r}) = N'(\bC^{r|r})$, which means it suffices to check the ascending chain condition for submodules of $M(\bC^{r|r})$. But this is a finitely generated module over a finitely generated supercommutative algebra $A(\bC^{r|r})$ (in the usual sense) and hence is noetherian. The same argument applies in the bivariate case.
\end{proof}

\begin{remark} \label{rmk:single-var}
The isomeric algebra $\Sym(\bV)$ is bounded since $\Sym^n(\bV) = \bT_n(\bV)$. It is therefore noetherian by the above proposition.
\end{remark}

\subsection{The isomeric algebra $A$} \label{ss:A}

Recall that $\bV=\bC^{\infty|\infty}$ with basis $\{e_i,f_i\}_{i \ge 1}$. We let $\bU=2^{-1}(\bV \otimes \bV)$. Recall that this is the $\zeta$-eigenspace of $\alpha \otimes \alpha$ acting on $\bV \otimes \bV$. Define
\begin{align*}
v_{i,j} &= (1+\zeta) e_i \otimes e_j+ (1-\zeta) f_i \otimes f_j \\
w_{i,j} &= (1+\zeta) e_i \otimes f_j + (1-\zeta) f_i \otimes e_j
\end{align*}
Then $\{v_{i,j},w_{i,j}\}_{i,j \ge 1}$ is a basis of $\bU$, with $v_{i,j}$ even and $w_{i,j}$ odd. The algebra $\fq \times \fq$ naturally acts on $\bV \otimes \bV$ and carries $\bU$ into itself.

Let $A=\Sym(\bU)$. This is a polynomial superalgebra on elements $\{x_{i,j},y_{i,j}\}_{i \ge 1}$, where $x_{i,j}$ corresponds to $v_{i,j}$ (and has degree~0) and $y_{i,j}$ to $w_{i,j}$ (and has degree~1). This is a bivariate isomeric algebra, and the main algebra of interest in this paper. By the isomeric analog of the Cauchy decomposition (\cite[Theorem 3.1]{chengwang2}, the infinite case follows by taking a direct limit), we have the irreducible decomposition
\begin{displaymath}
A = \bigoplus_{\lambda \in \Lambda} 2^{-\delta(\lambda)} (T_{\lambda} \otimes T_{\lambda}).
\end{displaymath}
Here $\delta(\lambda)$ is defined in \S \ref{ss:polrep}. In particular, we see that $A$ is multiplicity free as a $\fq \times \fq$-representation.

\section{Classification of ideals in $A$} \label{s:class}

\subsection{Overview} \label{ss:class-overview}

The goal of this section is to classify the ideals in the bivariate isomeric algebra $A$ that are stable by $\fq \times \fq$. We now explain the basic plan of attack (notation and definitions are explained below).
\begin{itemize}
\item Schur--Sergeev duality gives an equivalence of categories $\Mod_{\cC} = \Rep^{\pol}(\fq)$, where $\cC$ is a category built out of the Hecke--Clifford algebras. In fact, this is an equivalence of symmetric monoidal categories. There is a similar equivalence $\Mod_{\cC \boxtimes \cC}=\Rep^{\pol}(\fq \times \fq)$. It follows that $A=\Sym(\bU)$ corresponds to the $\cC \boxtimes \cC$ algebra $R'=\Sym(U)$, where $U$ corresponds to $\bU$. Thus it suffices to classify ideals in $R'$.
\item There is a natural equivalence $\cC^{\op} \cong \cC$ of symmetric monoidal categories that we call transpose. Let $R$ be the $\cC^{\op} \boxtimes \cC$ algebra obtained by applying transpose to the first factor in $R'$. Then it suffices to classify the ideals of $R$.
\item For any category $\cD$, there is a canonical $\cD^{\op} \boxtimes \cD$ module given by $(x,y) \mapsto \Hom_{\cD}(x,y)$. It turns out that $R$ is this canonical module for $\cC^{\op} \boxtimes \cC$. By general reasons, it follows that ideals of $R$ correspond to tensor ideals of $\cC$. It thus suffices to classify these.
\item We establish a general classification of tensor ideals in monoidal supercategories of a particular form (that encompasses $\cC$). One of the key assumptions is a certain form of the Pieri rule.
\end{itemize}
This section essentially follows the above plan in reverse. We thus begin with an abstract classification of tensor ideals and work our way back to $A$.

\subsection{Semisimple superalgebras}

Let $B$ be a finite dimensional semisimple $\bC$-superalgebra. By ``semisimple,'' we mean that the abelian category $\Mod_B^{\circ}$ is semisimple. Let $\Lambda(B)$ be the set of isomorphism classes of left $B$-modules, where we allow odd isomorphisms. For $\lambda \in \Lambda(B)$, we let $S_{\lambda}$ be a representative simple module. Thus every simple object of $\Mod_B^{\circ}$ is isomorphic to some $S_{\lambda}$ or $S_{\lambda}[1]$. We let $\wt{\rK}(B)$ be the Grothendieck group of the subcategory of $\Mod_B^{\circ}$ spanned by finite length modules, and we let $\rK(B)$ be the quotient of $\wt{\rK}(B)$ by the relations $[M]=[M[1]]$, i.e., a module and its shift represent the same class in $\rK(B)$. (The construction $\rK(-)$ is discussed in \cite[\S 2.6]{supermon}, where it is denoted $\rK_+(-)$.) Then $\rK(B)$ is the free abelian group on the classes $[S_{\lambda}]$ with $\lambda \in \Lambda(B)$. We let $J^{\lambda} \subset B$ be the $S_{\lambda}$-isotypic piece of $B$, i.e., the sum of all left ideals of $B$ isomorphic to $S_{\lambda}$ or $S_{\lambda}[1]$.

\begin{proposition} \label{prop:Jlambda}
We have the following:
\begin{enumerate}
\item $J^{\lambda}$ is a $2$-sided ideal of $B$.
\item $J^{\lambda}$ is simple as a $(B,B)$-bimodule.
\item $B = \bigoplus_{\lambda \in \Lambda(B)} J^{\lambda}$ as a $(B,B)$-bimodule.
\end{enumerate}
\end{proposition}

\begin{proof}
These results are proven for classical semisimple algebras in \cite[XVII.4, XVII.5]{Lang}. We just note that the proofs carry over to the super case as well.
\end{proof}

\begin{proposition} \label{prop:Jlambda2}
Let $J$ be a left ideal of $B$ and let $J'$ be the $2$-sided ideal it generates. Then $J'=\bigoplus_{\lambda \in S} J^{\lambda}$ where $S$ is the set of $\lambda \in \Lambda(B)$ for which $S_{\lambda}$ or $S_{\lambda}[1]$ appears as a simple constituent of $J$.
\end{proposition}

\begin{proof}
This follows from the fact that a semisimple superalgebra $B$ is isomorphic to a direct product of the endomorphism algebras of its simple modules and that each endomorphism algebra, as a left module, is isomorphic to a direct sum of copies of the corresponding simple (see \cite[\S 3.1]{chengwang}).
\end{proof}

\subsection{A classification of tensor ideals}

The goal of this section is to classify left tensor ideals in certain monoidal supercategories. We first recall the relevant definitions.

\begin{definition} \label{def:cat-ideal}
Let $\cC$ be a supercategory. An {\bf ideal} of $\cC$ is a rule $I$ assigning to each pair of objects $(x,y)$ a homogeneous subspace $I(x,y)$ of $\Hom_\cC(x,y)$ such that the following conditions hold:
\begin{itemize}
\item Given $f \in I(x,y)$ and $g \in \Hom_{\cC}(y,z)$, we have $g \circ f \in I(x,z)$.
\item Given $f \in I(x,y)$ and $h \in \Hom_{\cC}(w,x)$, we have $f \circ h \in I(w,y)$.
\end{itemize}
We note that $I(x,x)$ is a 2-sided homogeneous ideal of the superalgebra $\End_{\cC}(x)$ for every object $x$.
\end{definition}

\begin{definition} \label{def:tensor-ideal}
Let $(\cC, \amalg)$ be a monoidal supercategory. A {\bf left tensor ideal} of $\cC$ is an ideal $I$ of $\cC$ satisfying the following condition:
\begin{itemize}
\item Given $f \in I(x,y)$ and an object $z$, we have $\id_z \amalg f \in I(z \amalg x, z \amalg y)$.
\end{itemize}
We note that if $f \in I(x,y)$ and $g \in \Hom_{\cC}(x', y')$ then $g \amalg f \in I(x \amalg x', y \amalg y')$.
\end{definition}

Fix, for the remainder of this section, a monoidal supercategory $(\cC, \amalg)$. We now introduce a number of conditions on $\cC$.
\begin{itemize}
\item[(C1)] We have a bijection $\bN \to \Ob(\cC)$, denoted $n \mapsto [n]$, such that $[n] \amalg [m]=[n+m]$.
\item[(C2)] We have $\Hom_{\cC}([n], [m])=0$ for $n \ne m$.
\item[(C3)] The algebra $B_n=\End_{\cC}([n])$ is finite dimensional and semisimple.
\end{itemize}
We note that the monoidal operation induces a superalgebra homomorphism
\begin{displaymath}
\iota_{n,m} \colon B_n \otimes B_m \to B_{n+m}
\end{displaymath}
for all $n,m$. In fact, giving $\cC$, subject to (C1)--(C3), is equivalent to giving the sequence $(B_n)_{n \ge 0}$ of semisimple superalgebras together with the $\iota$ maps, such that certain conditions hold.

Let $\Lambda(\cC)=\coprod_{n \ge 0} \Lambda(B_n)$. Let $\rK(B_n)$ be the Grothendieck group discussed in the previous section, and let $\rK(\cC)=\bigoplus_{n \ge 0} \rK(B_n)$. We note that $\rK(\cC)$ is a free abelian group with basis $[S_{\lambda}]$ with $\lambda \in \Lambda$. The monoidal structure induces a ring structure on $\rK(\cC)$; precisely, if $M$ is a $B_m$-module and $N$ is a $B_n$-module then
\begin{displaymath}
[M] \cdot [N] = [B_{n+m} \otimes_{B_m \otimes B_n} (M \otimes N)],
\end{displaymath}
where the tensor product is formed via $\iota_{m,n}$. Note that this tensor product is exact in $M$ and $N$ since the $B$'s are semisimple, and thus is well-defined on the Grothendieck group.

Our next assumption is an analog of the Pieri rule:
\begin{itemize}
\item[(C4)] There is a partial order $\subset$ on $\Lambda(\cC)$ such that for any $\lambda \in \Lambda(B_n)$ we have
\begin{displaymath}
[B_1] \cdot [S_{\lambda}] = \sum_{\lambda \subset \mu,\, \mu \in \Lambda(B_{n+1})} c_{\lambda,\mu} [S_{\mu}]
\end{displaymath}
for \emph{positive} integers $c_{\lambda,\mu}$.
\end{itemize}
For $\lambda \in \Lambda_n$, we let $J^{\lambda} \subset B_n$ be the $S_{\lambda}$-isotypic piece of $B_n$, as in the previous section. If $I$ is a left tensor ideal of $\cC$ then $I_n=I([n],[n])$ is a 2-sided ideal of $B_n$; moreover, $I$ is determined by the $I_n$'s due to (C2).

\begin{theorem} \label{thm:ideal-structure1}
Suppose conditions (C1)--(C4) above hold. Let $\lambda \in \Lambda(B_n)$ and let $I^{\lambda}$ be the left tensor ideal of $\cC$ generated by $J^{\lambda}$. Then $I^{\lambda}_m=\bigoplus_{\lambda \subset \mu,\, \mu \in \Lambda(B_m)} J^{\mu}$ for any $m$.
\end{theorem}

Given a 2-sided ideal $J$ of $B_n$, we let $\Sigma(J)$ be the 2-sided ideal of $B_{n+1}$ generated by $\iota_{1,n}(1 \otimes J)$. The following is the key observation:

\begin{lemma} \label{lem:ideal-structure1-1}
Given $\lambda \in \Lambda_n$, we have $\Sigma(J^{\lambda})=\bigoplus_{\lambda \subset \mu,\, \mu \in \Lambda(B_{n+1})} J^{\mu}$.
\end{lemma}

\begin{proof}
Let $J$ be the left ideal of $B_{n+1}$ generated by $\iota_{1,n}(1 \otimes J^{\lambda})$. Then $\Sigma(J^{\lambda})$ is the 2-sided ideal generated by $J$. The natural map
\begin{displaymath}
B_{n+1} \otimes_{B_1 \otimes B_n} (B_1 \otimes J^{\lambda}) \to J
\end{displaymath}
is an isomorphism: indeed, it is surjective by definition, and it is injective since $\iota_{1,n}$ is flat. We thus have $[J]=[B_1] [J^{\lambda}]$. Since $[J^{\lambda}]=n [S_{\lambda}]$ for some $n>0$, we see that $[J]=n [B_1] [S_{\lambda}]$. By (C4), it follows that $S_{\mu}$ is a constituent of $J$ if and only if $\lambda \subset \mu$ and $\mu \in \Lambda(B_{n+1})$. Thus the 2-sided ideal generated by $J$ is $\bigoplus_{\mu \in \Lambda(B_{n+1})} J^{\mu}$ (Proposition~\ref{prop:Jlambda2}), which completes the proof.
\end{proof}

For $m \ge 0$, we let $\Sigma^m$ be the $m$-fold iterate of $\Sigma$.

\begin{lemma} \label{lem:ideal-structure1-2}
Given $\lambda \in \Lambda_n$, we have $\Sigma^m(J^{\lambda})=\bigoplus_{\lambda \subset \mu,\, \mu \in \Lambda(B_{n+m})} J^{\mu}$.
\end{lemma}

\begin{proof}
This follows from applying the previous lemma iteratively.
\end{proof}

\begin{proof}[Proof of Theorem~\ref{thm:ideal-structure1}]
Let $\lambda \in \Lambda(B_n)$ be given. Define $I'_m=\Sigma^{m-n}(J^{\lambda})$, which we take to be~0 for $m<n$. We claim that $I'$ is a left tensor ideal of $\cC$. Since $I'_m$ is a 2-sided ideal of $B_m$ for all $m$, it follows that $I'$ is closed under arbitrary compositions in $\cC$, that is, $I'$ is an ideal of $\cC$. Now, if $f \in B_m$ then $\id_{[k]} \amalg f=\iota_{k,n}(1 \otimes f)$. It follows that if $f$ belongs to a 2-sided ideal $J$ of $B_m$ then $\id_{[k]} \amalg f$ belongs to $\Sigma^k(J)$. In particular, if $f \in I'_m$ then $\id_{[k]} \amalg f \in I'_{m+k}$. Thus the claim holds.

Now, let $I^{\lambda}$ be the left tensor ideal generated by $J^{\lambda}$. Since $I'$ is a left tensor ideal containing $J^{\lambda}$, we have $I^{\lambda} \subset I'$. Since $I^{\lambda}$ is a left tensor ideal, we have $\Sigma(I^{\lambda}_m) \subset I^{\lambda}_{m+1}$ for all $m$, and so it follows (by a simple inductive argument) that $I' \subset I^{\lambda}$. We thus have $I^{\lambda}=I'$, and so $I^{\lambda}_m=\Sigma^{m-n}(J^{\lambda})$. Thus the theorem follows from Lemma~\ref{lem:ideal-structure1-2}.
\end{proof}

\subsection{Opposite categories}

Suppose that $B$ is a superalgebra. We define the {\bf opposite} superalgebra, denoted $B^{\op}$, as follows. The super vector space underlying $B^{\op}$ is just that underlying $B$. The multiplication $\bullet$ on $B^{\op}$ is defined on homogeneous elements $x$ and $y$ by $x \bullet y = (-1)^{\vert x \vert \vert y \vert} yx$.

Now suppose that $\cC$ is a supercategory. We define the {\bf opposite} supercategory, denoted $\cC^{\op}$, as follows. The objects of $\cC^{\op}$ are the same as the objects as $\cC$. We put $\Hom_{\cC^{\op}}(x,y)=\Hom_{\cC}(y,x)$. Given homogeneous morphisms $f \colon x \to y$ and $g \colon y \to z$, their composition in $\cC^{\op}$ is defined to be $(-1)^{\vert f \vert \vert g \vert} g \circ f$. We note that $\End_{\cC^{\op}}(x)=\End_{\cC}(x)^{\op}$.

\subsection{From tensor ideals to algebra ideals} \label{ss:R}

Let $\cC$ be a small supercategory. A {\bf $\cC$-module} is a superfunctor $\cC \to \SVec$. (A superfunctor is just a functor that is $\SVec^{\circ}$-enriched, i.e., it preserves homogeneity and degree of morphisms.) We let $\Mod_{\cC}$ be the supercategory of $\cC$-modules. The category $\Mod_{\cC}^{\circ}$ is abelian.

Suppose now that $\cC$ has a monoidal structure $\amalg$. Then $\Mod_{\cC}$ admits a natural monoidal structure $\otimes$, namely Day convolution. Thus $\Mod_{\cC}^{\circ}$ is a monoidal category, and so we can speak of algebras, modules over algebras, ideals in algebras, and so on. We note that to give an algebra in $\Mod^{\circ}_{\cC}$ amounts to giving a $\cC$-module $B$ with even maps $B(x) \otimes B(y) \to B(x \amalg y)$ for all objects $x$ and $y$, satisfying various conditions.

The supercategory $\cC^{\op} \boxtimes \cC$ is naturally monoidal. (Here $\boxtimes$ denotes the product of enriched categories, see \cite[\S 2.5]{supermon}.) This category admits a module $R$ defined by $R(x,y)=\Hom_{\cC}(x,y)$. Furthermore, the module $R$ admits an algebra structure: the map
\begin{displaymath}
R(x,y) \otimes R(x',y') \to R(x \amalg x', y \amalg y')
\end{displaymath}
is simply given by $f \otimes g \mapsto f \amalg g$.

\begin{proposition} \label{prop:ideal-corr}
There is a natural bijective correspondence
\begin{displaymath}
\{ \text{left ideals of $R$} \} \longleftrightarrow \{ \text{left tensor ideals of $\cC$} \}.
\end{displaymath}
\end{proposition}

\begin{proof}
Suppose that $I$ is a left ideal of $R$. Then $I$ is in particular a $\cC$-submodule of $R$, i.e., a sub superfunctor of $R$. This exactly means that $I$ is an ideal of the category $\cC$ in the sense of Definition~\ref{def:cat-ideal}. Since $I$ is a left ideal of $R$, we see that if $f \in R(x,y)$ and $g \in I(x',y')$ then $f \amalg g$ belongs to $I(x \amalg x', y \amalg y')$. This shows that $I$ is a left tensor ideal of $\cC$. The above reasoning is reversible, and so the proposition follows.
\end{proof}

\subsection{The Hecke--Clifford algebras}

The Clifford algebra $\Cl_n$ is the superalgebra generated by odd-degree elements $\alpha_1,\dots,\alpha_n$ subject to the relations $\alpha_i^2=1$ and $\alpha_i \alpha_j = -\alpha_j \alpha_i$ for $i \ne j$. The symmetric group $\fS_n$ acts on $\Cl_n$ by $\sigma \cdot \alpha_i = \alpha_{\sigma(i)}$ and we define the Hecke--Clifford algebra $\cH_n$ to be the semi-direct product $\bC[\fS_n] \ltimes \Cl_n$. This is again a superalgebra, where we define $|\sigma|=0$ for $\sigma \in \fS_n$.

The algebra $\cH_n$ is semisimple \cite[(3.25)]{chengwang}. The isomorphism classes of simple $\cH_n$-modules are parametrized by strict partitions of $n$ \cite[Proposition 3.41]{chengwang}. We thus identify $\Lambda_n=\Lambda(\cH_n)$ with the set of such partitions. Given $\lambda \in \Lambda_n$, we have an even isomorphism $S_{\lambda} \cong S_{\lambda}[1]$ if and only if $\delta(\lambda)=1$ \cite[Corollary 3.44]{chengwang}.

We define a linear map $(-)^{\dag} \colon \cH_n \to \cH_n$, called {\bf transpose}, by the formula
\begin{displaymath}
(\alpha_{i_1} \cdots \alpha_{i_k} \sigma)^{\dag}
=\zeta^{k^2} \sigma^{-1} \alpha_{i_k} \cdots \alpha_{i_1}
=(-1)^{\binom{k}{2}} \zeta^k \sigma^{-1} \alpha_{i_k} \cdots \alpha_{i_1},
\end{displaymath}
where here $\sigma \in \fS_n$. One readily verifies that this is well-defined.

\begin{proposition}
The transpose map defines an isomorphism $\cH_n^{\op} \to \cH_n$ of superalgebras.
\end{proposition}

\begin{proof}
We must show that $(xy)^{\dag}=(-1)^{\vert x \vert \vert y \vert} y^{\dag} x^{\dag}$ for homogeneous elements $x,y \in \cH_n$. If $x$ is a basis vector (i.e., of the form $\alpha_{i_1} \cdots \alpha_{i_k} \sigma$) then we have $x^{\dag}=\zeta^{\vert x \vert^2} x^{-1}$. Suppose that $x$ and $y$ are basis vectors of degrees $n$ and $m$. Then $xy$ is a basis vector of degree $n+m$. We have
\begin{align*}
(xy)^{\dag} &= \zeta^{(n+m)^2} (xy)^{-1} \\
(-1)^{nm} y^{\dag} x^{\dag} &= \zeta^{2nm+n^2+m^2} y^{-1} x^{-1}.
\end{align*}
These are equal. We have thus verified the formula when $x$ and $y$ are basis vectors, from which the general case follows.
\end{proof}

Suppose that $M$ is an $\cH_n$-module. Then the dual vector space $M^*$ is naturally an $\cH_n^{\op}$-module. Composing with the transpose isomorphism, we can regard $M^*$ as an $\cH_n$-module; we denote this module by $M^{\dag}$.

\subsection{The Hecke--Clifford category}

We now assemble the Hecke--Clifford algebras into a supercategory $\cC$. The objects of $\cC$ are formal symbols $[n]$ with $n \in \bN$. We put $\End_{\cC}([n])=\cH_n$ and $\Hom_{\cC}([n], [m])=0$ for $n \ne m$. Composition is defined in the obvious manner.

We now define a monoidal structure $\amalg$ on $\cC$. On objects, we put $[m] \amalg [n]=[m+n]$. For $m,n \ge 0$, we define
\begin{displaymath}
\iota_{m,n} \colon \cH_m \otimes \cH_n \to \cH_{m+n}
\end{displaymath}
by $\alpha_i \otimes 1 \mapsto \alpha_i$, $1 \otimes \alpha_j \mapsto \alpha_{m+j}$, and using the usual embedding $\fS_m \times \fS_n \subset \fS_{m+n}$. One readily verifies that this is a homomorphism of superalgebras. The monoidal operation $\amalg$ is defined on morphisms using $\iota_{m,n}$.

We now define a symmetry $\beta$ on the monoidal structure $\amalg$. Let $\tau_{m,n} \in \fS_{m+n}$ be the permutation defined by
\begin{displaymath}
\tau_{m,n}(i) = \begin{cases}
i+n & \text{if $1 \le i \le m$} \\
i-m & \text{if $m+1 \le i \le m+n$} \end{cases}.
\end{displaymath}
One readily verifies that for homogeneous elements $x \in \cH_m$ and $y \in \cH_n$ we have
\begin{displaymath}
\tau_{m,n} \iota_{m,n}(x \otimes y) \tau_{m,n}^{-1} = (-1)^{mn} \iota_{n,m}(y \otimes x).
\end{displaymath}
We define $\beta_{m,n} \colon [m] \amalg [n] \to [n] \amalg [m]$ to be the element $\tau_{m,n} \in \End_{\cC}([m+n])$.

Put $\Lambda=\Lambda(\cC)$, which we identify with the set of all strict partitions. The group $\rK(\cC)$ is the free abelian group on the classes $[S_{\lambda}]$ with $\lambda \in \Lambda$. We let $\subset$ be the partial order on $\Lambda$ given by containment of Young diagrams; explicitly, $\lambda \subseteq \mu$ means $\lambda_i \le \mu_i$ for all $i$. We have the following version of the Pieri rule (recall that $\delta(\lambda)$ is 0 or 1 depending on if $\ell(\lambda)$ is even or odd, respectively):

\begin{proposition} \label{prop:hecke-pieri}
Let $\lambda$ be a strict partition of $n$. We then have
\begin{displaymath}
[\cH_1] [S_\lambda] =  \sum_{\substack{\mu \supset \lambda \\ \ell(\mu)  =\ell(\lambda) }} 2 [S_\mu] + \sum_{\substack{\mu \supset \lambda \\ \ell(\mu)  =\ell(\lambda) +1 }} 2^{\delta(\lambda)} [S_\mu]
\end{displaymath}
in $\rK(\cC)$, where $\mu$ varies over strict partitions of $n+1$.
\end{proposition}

\begin{proof}
  Let $Q_\lambda$ be the Schur $Q$-functions indexed by strict partitions $\lambda$, see \cite[\S A.3]{chengwang}. They are linearly independent and their $\bZ$-span is a subring $\Gamma$. There is a characteristic map $\ch \colon \rK(\cC) \otimes {\bQ} \to \Gamma \otimes {\bQ}$, which is an isomorphism of graded algebras (see \cite[\S 3.3]{chengwang}). We have $\ch([S_\lambda]) = 2^{(\delta(\lambda)-\ell(\lambda))/2} Q_{\lambda}$ and hence $\ch([\cH_1]) = Q_1$ \cite[\S 3.3.6]{chengwang}. Since $\ch$ is an isomorphism, the result follows from the following identity \cite[III.8.7 and III.8.15]{macdonald}
\[
  Q_1 Q_{\lambda} =  \sum_{\substack{\mu \supset \lambda \\ \ell(\mu)  =\ell(\lambda) }} 2 Q_{\mu} + \sum_{\substack{\mu \supset \lambda \\ \ell(\mu)  =\ell(\lambda) +1 }} Q_{\mu}. \qedhere
\]
\end{proof}

\begin{proposition}
The category $\cC$ satisfies (C1)--(C4).
\end{proposition}

\begin{proof}
Conditions (C1) and (C2) are essentially by definition, and (C3) was discussed in the previous section. Condition (C4) follows from Proposition~\ref{prop:hecke-pieri}.
\end{proof}

We define the transpose functor $(-)^{\dag} \colon \cC^{\op} \to \cC$ to be the identity on objects and the transpose operation $(-)^{\dag}$ on each algebra $\cH_n$.

\begin{proposition}
The transpose functor is naturally an equivalence of symmetric monoidal categories.
\end{proposition}

\begin{proof}
It is clear that the transpose functor is indeed a functor and an equivalence. To define a monoidal structure on the transpose functor, we take the isomorphism $([n] \amalg [m])^{\dag} \to [n]^{\dag} \amalg [m]^{\dag}$ to be the identity map; note that both objects are $[n+m]$. For this to define a monoidal structure, we need
\begin{displaymath}
\iota_{m,n}(f \otimes g)^{\dag}=\iota_{m,n}(f^{\dag} \otimes g^{\dag})
\end{displaymath}
for $f \in \cH_m$ and $g \in \cH_n$. Both sides define algebra homomorphisms $\cH_m \otimes \cH_n \to \cH_{n+m}^{\op}$, so it suffices to check the equality on algebra generators, where it is clear. Finally, for the symmetry condition, we need $(\beta^{\op}_{m,n})^{\dag}=\beta_{m,n}$, where $\beta^{\op}$ denotes the symmetry in $\cC^{\op}$. By definition, the morphism $\beta^{\op}_{m,n}$ in $\cC^{\op}$ is equal to the morphism $\beta_{n,m}$ in $\cC$. Thus the condition we require is $\tau_{n,m}^{\dag}=\tau_{m,n}$, which is true.
\end{proof}

Let $M$ be a $\cC$-module. Then $x \mapsto M(x)^*$ is a $\cC^{\op}$-module. Composing with the transpose equivalence, we get a $\cC$-module that we denote by $M^{\dag}$. If we regard $M$ as a sequence $(M_n)_{n \ge 0}$, where $M_n$ is an $\cH_n$-module, then $M^{\dag}$ is just the sequence $(M^{\dag}_n)_{n \ge 0}$, where $M_n^{\dag}$ is as in the previous section. The construction $(-)^{\dag}$ defines an equivalence $\Mod_{\cC}^{\op} \to \Mod_{\cC}$ that is compatible with tensor products. It thus induces a ring automorphism of $\rK(\cC)$ that we denote by $(-)^{\dag}$.

\begin{proposition}
We have $(-)^{\dag}=\id$ on $\rK(\cC)$.
\end{proposition}

\begin{proof}
Let $V_n=\Cl_n$, which is naturally a simple left $\cH_n$-module \cite[\S 3.3.5]{chengwang}. Under the characteristic map (see the proof of Proposition~\ref{prop:hecke-pieri}), the class $[V_n]$ in $\rK(\cC)$ corresponds to the function $Q_{(n)}$ \cite[Lemma~3.40]{chengwang}. Since the $Q_{(n)}$ generate $\Gamma \otimes \bQ$ as a $\bQ$-algebra, it follows that the $[V_n]$ generate $\rK(\cC) \otimes \bQ$ as a $\bQ$-algebra. It thus suffices to show that $[V_n]^{\dag}=[V_n]$. For this, it is enough to show that $V_n^{\dag}$ and $V_n$ are isomorphic as $\cH_n$-modules.

Let $t \colon \Cl_n \to \bC$ be the map satisfying $t(1)=1$ and $t(m)=0$ for any monomial $m \ne 1$ in the $\alpha_i$'s. This map satisfies $t(xy)=t(yx)$ for $x,y \in \Cl_n$ (there is no sign). For $x \in V_n$, let $\lambda_x \in V_n^*$ be defined by $\lambda_x(y)=t(x y^{\dag})$. We thus have a linear map $\lambda \colon V_n \to V_n^*$ via $x \mapsto \lambda_x$, which is easily seen to be an isomorphism of super vector spaces. Recall that $V_n^*$ is a left $\cH_n$-module via $(a \lambda)(y)=\lambda(a^{\dag} y)$. For $a \in \Cl_n$, we have
\begin{displaymath}
\lambda_{ax}(y)=t(ax y^{\dag})=t(x y^{\dag} a)=t(x (a^{\dag} y)^{\dag})=\lambda_x(a^{\dag} y)=(a \lambda_x)(y),
\end{displaymath}
and so $\lambda_{ax}=a \lambda_x$. We thus see that $\lambda$ is a homomorphism of left $\Cl_n$-modules. Since $t$ is $\fS_n$-invariant, it follows that $\lambda$ is also $\fS_n$-equivariant. We thus see that $\lambda$ is an isomorphism of left $\cH_n$-modules, which completes the proof.
\end{proof}

\begin{corollary} \label{cor:Slambda-dag}
We have $S_{\lambda}^{\dag} \cong S_{\lambda}$ as $\cH_n$-modules for all $\lambda \in \Lambda_n$.
\end{corollary}

Let $R$ be the canonical $\cC^{\op} \boxtimes \cC$ algebra defined in \S \ref{ss:R}. Let $R'$ be the $\cC \boxtimes \cC$ algebra obtained by applying transpose to the first argument in $R$. Let $U$ be the quotient of $\Cl_1 \otimes \Cl_1$ by the 2-sided ideal generated by $\alpha_1 \otimes \alpha_1 - \zeta$. One easily sees that $R'([1],[1]) \cong U$.

\begin{proposition} \label{prop:R=SymU}
The identification $U=R([1],[1])$ induces an isomorphism $\Sym(U) \to R'$ of $\cC \boxtimes \cC$ algebras.
\end{proposition}

\begin{proof}
We first construct an isomorphism $\phi_n \colon \Sym^n(U) \to R'([n], [n])$ of $\cH_n \otimes \cH_n$-modules for each $n$.  For simplicity of notation, we denote $\cH_n \otimes \cH_n$ by $C_n$. We note that $R'([n],[n])$ can be identified with the $C_n$-module $\cH_n$ given on generators by
\begin{displaymath}
(a \otimes b).x = (-1)^{|a|(|x|+|b|)}bxa^{\dag}.
\end{displaymath}
In particular, we identify $U$ with the $C_1$-module $\cH_1$. Thus, the $C_n$-module $\Sym^n(U)$   is given by the $\fS_n$-coinvariants of $C_{n} \otimes_{C_{1}^{\otimes n}} \cH_{1}^{\otimes n}$ where  $C_{1}^{\otimes n} \to C_{n}$ is the natural map that identifies  $C_{1}^{\otimes n}$ with the subalgebra $\Cl_n \otimes \Cl_n$ of $C_{n}$. Under the latter identification,  $\cH_{1}^{\otimes n}$ is isomorphic to the submodule $\Cl_n$ of the $C_n$-module $\cH_n$. Thus we have $C_{n} \otimes_{C_{1}^{\otimes n}} \cH_{1}^{\otimes n} \cong C_{n} \otimes_{\Cl_n \otimes \Cl_n} \Cl_{n}$ as $C_n$-modules. As a complex vector space, $C_{n} \otimes_{\Cl_n \otimes \Cl_n} \Cl_{n}$ has a basis of elements of the form $\sigma \otimes \tau \otimes \alpha_1^{\epsilon_1} \cdots \alpha_n^{\epsilon_n}$ where $\sigma, \tau \in \fS_n$, $\epsilon_i \in \{0,1\}$. The action of $\fS_n$ on the $n$th tensor power $C_{n} \otimes_{\Cl_n \otimes \Cl_n} \Cl_{n}$ is given by 
	 \[
	 g.(\sigma \otimes \tau \otimes \alpha_1^{\epsilon_1} \cdots \alpha_n^{\epsilon_n}) =  \sigma g^{-1} \otimes \tau g^{-1} \otimes (\alpha_1^{\epsilon_1} \cdots \alpha_n^{\epsilon_n})^g,
	 \]
	 where $g \in \fS_n$. Thus its $\fS_n$-coinvariants, $\Sym^n(U)$, is the free complex vector space on symbols of the form $1 \otimes \tau \otimes \alpha_1^{\epsilon_1} \cdots \alpha_n^{\epsilon_n}$. Let $\phi_n \colon \Sym^n(U) \to \cH_n$ be the $\bC$-linear isomorphism given by
	 \[
	 1 \otimes \tau \otimes \alpha_1^{\epsilon_1} \cdots \alpha_n^{\epsilon_n} \mapsto  \tau  \alpha_1^{\epsilon_1} \cdots \alpha_n^{\epsilon_n}.
	 \]
We show that $\phi_n$ is $C_n$-linear. Denote $ \alpha_1^{\epsilon_1} \cdots \alpha_n^{\epsilon_n}$ by $c$, and let $c_1 \sigma_1, c_2\sigma_2$ be two basis elements of $\cH_n$. We have 
\begin{align*}
	 	(c_1 \sigma_1 \otimes c_2 \sigma_2)(1 \otimes \tau \otimes c)  &= (c_1 \sigma_1 \otimes  c_2 \sigma_2 \tau \otimes c) \\
	 	& = (\sigma_1 \otimes \sigma_2 \tau)(c_1^{\sigma_1^{-1}} \otimes c_2^{(\sigma_2 \tau)^{-1}} \otimes c)\\
	 	& = \zeta^{|c_1|^2} (-1)^{|c_1|(|c|+|c_2|)} \sigma_1 \otimes \sigma_2 \tau \otimes c_2^{(\sigma_2 \tau)^{-1}}  c (c_1^{-1})^{\sigma_1^{-1}} \\
	 	&= \zeta^{|c_1|^2} (-1)^{|c_1|(|c| +|c_2|)} 1 \otimes \sigma_2 \tau \sigma_1^{-1} \otimes c_2^{\sigma_1(\sigma_2 \tau)^{-1}}  c^{\sigma_1} c_1^{-1},
	 \end{align*}
	 which under $\phi_n$ maps to $\zeta^{|c_1|^2} (-1)^{|c_1|(|c|+|c_2|)}   c_2  c^{\sigma_2 \tau } (c_1^{-1})^{\sigma_2 \tau \sigma_1^{-1} } \sigma_2 \tau \sigma_1^{-1}$. On the other hand in $\cH_n$, we have \begin{align*}
	 	(c_1 \sigma_1 \otimes c_2 \sigma_2) \phi(1 \otimes \tau \otimes c)  & = (c_1 \sigma_1 \otimes c_2 \sigma_2) c^{\tau} \tau \\
	 	&= \zeta^{|c_1|^2} (-1)^{|c_1|(|c|+|c_2|)} c_2 \sigma_2 c^{\tau} \tau \sigma_1^{-1} c_1^{-1} \\
	 	& = \zeta^{|c_1|^2} (-1)^{|c_1|(|c|+|c_2|)}  c_2 c^{\sigma_2 \tau } (c_1^{-1})^{\sigma_2 \tau \sigma_1^{-1}}\sigma_2 \tau \sigma_1^{-1}.
	 \end{align*}
Thus shows that $\phi_n$ is $C_n$-linear.

Now, both $\Sym(U)$ and $R'$ are concentrated on the diagonal of $\cC \boxtimes \cC$, that is, they are only non-zero on objects of the form $([n], [n])$. Thus the $\phi_n$'s define an isomorphism $\phi \colon \Sym(U) \to R'$ of $\cC \boxtimes \cC$ modules. To complete the proof, we simply observe that $\phi$ is an algebra homomorphism. Indeed, suppose that $x \in \Sym^n(U)$ and $y \in \Sym^m(U)$. Write $x=1 \otimes \sigma \otimes \alpha_1^{\epsilon_1} \cdots \alpha_n^{\epsilon_n}$ and $y=1 \otimes \tau \otimes \alpha_1^{\delta_1} \cdots \alpha_m^{\delta_m}$. Then one finds
\begin{displaymath}
xy=1 \otimes \sigma \tau \otimes \alpha_1^{\epsilon_1} \cdots \alpha_n^{\epsilon_n} \alpha_{n+1}^{\delta_1} \cdots \alpha_{n+m}^{\delta_m},
\end{displaymath}
(where here $\tau$ is regarded as a permutation of $\{n+1,\ldots,n+m\}$) and so
\begin{displaymath}
\phi_{n+m}(xy)=\sigma \tau \alpha_1^{\epsilon_1} \cdots \alpha_n^{\epsilon_n} \alpha_{n+1}^{\delta_1} \cdots \alpha_{n+m}^{\delta_m},
\end{displaymath}
which is exactly $\phi_n(x) \phi_m(y)=\iota_{n,m}(\phi_n(x),\phi_m(y))$.
\end{proof}

Given two $\cC$-modules $M$ and $N$, we let $M \boxtimes N$ denote the $\cC \boxtimes \cC$ module given by $(x,y) \mapsto M(x) \otimes N(y)$. If $M$ is an $\cH_n$-module, we regard $M$ as a $\cC$-module that is~0 outside of degree $n$. We now come to the main result of this section:

\begin{theorem} \label{thm:ideal-structure2}
The $\cC \boxtimes \cC$ algebra $\Sym(U)$ decomposes as $\bigoplus 2^{-\delta(\lambda)} (S_{\lambda} \boxtimes S_{\lambda})$, where the sum is over all strict partitions $\lambda$. Moreover, if $I^{\lambda}$ denotes the ideal of $\Sym(U)$ generated by the $\lambda$ summand then $I^{\lambda}$ contains the $\mu$ summand if and only if $\lambda \subset \mu$.
\end{theorem}

\begin{proof}
The algebra $\cH_n$ is naturally an $(\cH_n,\cH_n)$-bimodule, and as such it decomposes as $\bigoplus_{\lambda \in \Lambda_n} J^{\lambda}$ (Proposition~\ref{prop:Jlambda}). We can regard any $(\cH_n, \cH_n)$-bimodule as a left module over $\cH_n^{\op} \otimes \cH_n$, and thus as a $\cC^{\op} \boxtimes \cC$ module concentrated in degree $(n,n)$. With this convention, we have the decomposition of $\cC^{\op} \boxtimes \cC$ modules
\begin{displaymath}
R=\bigoplus_{n \ge 0} \cH_n=\bigoplus_{\lambda \in \Lambda} J^{\lambda}.
\end{displaymath}
The first identification above is essentially the definition of $R$, while the second follows from the decomposition of $\cH_n$. It follows from Theorem~\ref{thm:ideal-structure1} and Proposition~\ref{prop:ideal-corr} that the ideal of $R$ generated by $J^{\lambda}$ contains exactly those $J^{\mu}$ with $\lambda \subset \mu$.

As stated above, $J^{\lambda}$ can be regarded as a left module over $\cH_n^{\op} \otimes \cH_n$. As such, it is isomorphic to the submodule of $\End(S_{\lambda})=S_{\lambda}^* \otimes S_{\lambda}$ consisting of endomorphisms that supercommute with $\End_{\cH_n}(S_{\lambda})$; this follows from \cite[(3.4)]{chengwang}. If $\delta(\lambda)=0$ then $\End_{\cH_n}(S_{\lambda})=\bC$, and we find $J^{\lambda} \cong S_{\lambda}^* \otimes S_{\lambda}$. Suppose now that $\delta(\lambda)=1$. Then there exists an isomorphism $S_{\lambda} \to S_{\lambda}[1]$. From this, we find that there is an odd isomorphism $\alpha \colon S_{\lambda} \to S_{\lambda}$ that squares to~$1$ (note that any odd isomorphism squares to a scalar, by the super version of Schur's lemma \cite[Lemma~3.4]{chengwang}, and can thus be normalized to square to~1). The dual map $\alpha^* \colon S_{\lambda}^* \to S_{\lambda}^*$ squares to $-1$. We see that $J^{\lambda}$ is isomorphic to the $1$-eigenspace of $\alpha^* \otimes \alpha$ on $S_{\lambda}^* \otimes S_{\lambda}$. This coincides with the $\zeta$-eigenspace of $(\zeta \alpha^*) \otimes \alpha$. As $\zeta \alpha^*$ endows $S_{\lambda}^*$ with an isomeric structure, we see that $J^{\lambda}$ is isomorphic to the half tensor product $2^{-1} (S_{\lambda}^* \otimes S_{\lambda})$. Thus, in all cases, we see that $J^{\lambda} \cong 2^{-\delta(\lambda)} (S_{\lambda}^* \otimes S_{\lambda})$.

Via transpose on the first factor, we can regard $J^{\lambda}$ as an $\cH_n \otimes \cH_n$ module. As such, we see that it is isomorphic to $2^{-\delta(\lambda)}(S_{\lambda}^{\dag} \otimes S_{\lambda})$, which (by Corollary~\ref{cor:Slambda-dag}) is isomorphic to $2^{-\delta(\lambda)}(S_{\lambda} \otimes S_{\lambda})$. Combining this observation with the decomposition in the first paragraph, we find
\begin{displaymath}
R' = \bigoplus_{\lambda \in \Lambda} 2^{-\delta(\lambda)} (S_{\lambda} \boxtimes S_{\lambda}).
\end{displaymath}
Furthermore, we see that the ideal generated by the $\lambda$ summand contains the $\mu$ summand if and only if $\lambda \subset \mu$. The theorem now follows from the identification $R'=\Sym(U)$ (Proposition~\ref{prop:R=SymU}).
\end{proof}

\subsection{Application to the isomeric algebra $A$}

Recall that $\bV=\bC^{\infty|\infty}$ is our fixed infinite dimensional isomeric vector space.  Maintaining the notation from \S \ref{ss:A}, we let $\bU=2^{-1}(\bV \otimes \bV)$, which is a polynomial representation of $\fq \times \fq$, and we let $A=\Sym(\bU)$, a bivariate isomeric algebra. As noted in \S \ref{ss:A}, we have a Cauchy decomposition
\begin{align} \label{eqn:Q-cauchy}
A = \bigoplus_{\lambda \in \Lambda} 2^{-\delta(\lambda)} (T_{\lambda} \otimes T_{\lambda}),
\end{align}
where $\Lambda$ is the poset of strict partitions, as above, and $\delta(\lambda)$ is as in \S \ref{ss:polrep}. The follow theorem, which classifies the equivariant ideals of $A$, is one of the main results of this paper:

\begin{theorem} \label{thm:ideal-structure} 
Suppose $\lambda$ is a strict partition, and let $I^{\lambda}$ be the ideal of $A$ generated by $2^{-\delta(\lambda)} (T_{\lambda} \otimes T_{\lambda})$. Then we have $I^{\lambda} = \bigoplus_{\lambda \subseteq \mu} 2^{-\delta(\mu)} (T_{\mu} \otimes T_{\mu})$ where $\mu$ varies over strict partitions.
\end{theorem}

\begin{proof}
One formulation of Schur--Sergeev duality (see \cite[\S 8.2]{supermon}) asserts that there is an equivalence of symmetric monoidal supercategories $\Mod_{\cC} \to \Rep^{\pol}(\fq)$ satisfying $S_{\lambda} \mapsto T_{\lambda}$. Here $\Mod_{\cC}$ is endowed with the tensor product induced by the monoidal structure on $\cC$ (Day convolution). It follows that we have an equivalence of symmetric monoidal supercategories
\begin{displaymath}
\Phi \colon \Mod_{\cC \boxtimes \cC} \to \Rep^{\pol}(\fq \times \fq)
\end{displaymath}
satisfying $\Phi(S_{\lambda} \boxtimes S_{\mu})=T_{\lambda} \otimes T_{\mu}$. Since $U=2^{-1}(S_{(1)} \boxtimes S_{(1)})$, it follows that $\bU=\Phi(U)$, and $A=\Phi(\Sym(U))$. The result now follows from Theorem~\ref{thm:ideal-structure2}.
\end{proof}

Theorem~\ref{thm:ideal-structure} has an important consequence:

\begin{corollary} \label{cor:A-bounded}
If $I \subset A$ is a nonzero ideal, then $A/I$ is bounded.
\end{corollary}
      
\begin{proof}
  Since $I \ne 0$, there exists a strict partition $\lambda$ such that $I^\lambda \subseteq I$. Suppose that $\mu$ is a strict partition such that $\ell(\mu) \ge \lambda_1$. Then $\mu_i \ge \ell(\mu)-i+1 \ge \lambda_1-i+1 \ge \lambda_i$ and hence $\lambda \subseteq \mu$. So Theorem~\ref{thm:ideal-structure} implies that $\ell(A/I) < \lambda_1$.
\end{proof}

\begin{remark} \label{rmk:detideal}
Let $X=\Spec(A(\bC^{n|n},\bC^{m|m}))$, which we regard as a supervariety (i.e.,  a functor from superalgebras to sets). We can identify points of $X$ with linear maps $\bC^{n|n} \to \bC^{m|m}$ commuting with the isomeric structures. The image of any such map is isomorphic to $\bC^{r|r}$ for some $r$, which we refer to as the rank. Let $I_r \subset A$ be the sum of the $\lambda$ summands in the Cauchy decomposition with $\ell(\lambda)>r$. Then $I_r$ is an equivariant ideal of $A$. It follows from properties of the $\bT_{\lambda}$ (namely, that $\bT_{\lambda}(\bC^{r|r})=0$ if and only if $\ell(\lambda)>r$) that $V(I_r(\bC^{n|n}, \bC^{m|m}))$ is the closed subvariety of $X$ consisting of maps of rank $\le r$. Thus $I_r$ is an isomeric analog of the classical determinantal ideal. It follows from Theorem~\ref{thm:ideal-structure} that $I_r$ is generated by the $\lambda$ summand of the Cauchy decomposition, where $\lambda$ is the ``staircase'' partition $(r+1, r, \ldots, 2, 1)$.
\end{remark}

\section{Local structure of $A$-modules} \label{s:local}

\subsection{Statement of results}

Let $\fm$ be the ideal of $A$ generated by $x_{i,j}-\delta_{i,j}$ and $y_{i,j}$ for all $i$ and $j$. In this section, we study the local structure of $A$-modules at $\fm$. Our main result is:

\begin{theorem} \label{thm:loc}
Let $M$ be an $A$-module. Then there is a canonical and functorial isomorphism $M_{\fm} \to M/\fm M \otimes_{\bC} A_{\fm}$ of $\vert A_{\fm} \vert$-modules. In particular, $M_{\fm}$ is a free $\vert A_{\fm} \vert$-module.
\end{theorem}

The proof will take the rest of the section. The arguments are similar to those used in \cite[\S 4]{symsp1}, \cite[\S 3.2]{sym2noeth}, and \cite[\S 5]{periplectic}. Theorem~\ref{thm:loc} is a key result used to study generic $A$-modules in the following section.

\subsection{The algebras $\fh$ and $\fk$}

Let $\fh$ be the subalgebra of $\fq \times \fq$ consisting of all pairs of the form $(\tau g, g)$, where $\tau$ is the Chevalley automorphism of $\fq \times \fq$. Note that $\fh$ is isomorphic to $\fq$ (via the first projection). Let $\fb$ (resp.\ $\fn$) be the subalgebra of $\fq$ consisting of matrices of the form
\begin{displaymath}
\begin{pmatrix} a & b \\ -b & a \end{pmatrix}.
\end{displaymath}
with $a$ and $b$ upper triangular (resp.\ strictly upper triangular), and let $\fk=\fb \times \fn$, regarded as a subalgebra of $\fq \times \fq$.

\begin{proposition}
We have $\fq \times \fq=\fh \oplus \fk$.
\end{proposition}

\begin{proof}
To simplify notation, put
\begin{displaymath}
\{a,b\} = \begin{pmatrix} a & b \\ -b & a \end{pmatrix}.
\end{displaymath}
Let $a_1$, $a_2$, $b_1$, and $b_2$ be given matrices. We must show that there are unique matrices $c_1$, $c_2$, $d_1$, $d_2$, $e_1$, $e_2$, with $d_1,d_2$ upper triangular and $e_1,e_2$ strictly upper triangular such that
\begin{displaymath}
(\{a_1,a_2\}, \{b_1,b_2\}) = (\{c_1,c_2\}, \tau^{-1} \{c_1, c_2\}) + (\{d_1,d_2\}, \{e_1,e_2\}).
\end{displaymath}
The above equation is equivalent to
\begin{align*}
a_1 &= c_1+d_1 & b_1 &= -c_1^t+e_1 \\
a_2 &= c_2+d_2 & b_2 &= \zeta c_2^t+e_2.
\end{align*}
Apply transpose to the right two equations and multiply the bottom right equation by $\zeta$; then add the left column to the right. We obtain
\begin{align*}
a_1 &= c_1+d_1 & a_1+b_1^t &= d_1+e^t_1 \\
a_2 &= c_2+d_2 & a_2+\zeta b_2^t &= d_2+\zeta e^t_2
\end{align*}
The two right equations have unique solutions, as every matrix can be written uniquely as the sum of an upper triangular matrix and a strictly lower triangular matrix. The two left equations then obvious admit unique solutions as well.
\end{proof}

\begin{proposition}
The ideal $\fm$ is $\fh$-stable. 
\end{proposition}

\begin{proof}
Let $E_{i,j}$ be the elementary matrix whose $(i,j)$ entry is~1 and all other entries are~0; thus $E_{i,j}$ maps the $j$th basis vector to the $i$th basis vector. Let
\begin{displaymath}
X_{i,j} = \begin{pmatrix} E_{i,j} & 0 \\ 0 & E_{i,j} \end{pmatrix}, \qquad
Y_{i,j} = \begin{pmatrix} 0 & E_{i,j} \\ -E_{i,j} & 0 \end{pmatrix}.
\end{displaymath}
The $X_{i,j}$ and $Y_{i,j}$ elements form a basis for $\fq$. We have
\begin{align*}
X_{i,j}(e_k) &= \delta_{j,k} e_i & Y_{i,j}(e_k) &= -\delta_{j,k} f_i \\
X_{i,j}(f_k) &= \delta_{j,k} f_i & Y_{i,j}(f_k) &= \delta_{j,k} e_i.
\end{align*}
The elements $X'_{i,j}=(X_{i,j}, \tau^{-1} X_{i,j})$ and $Y'_{i,j}=(Y_{i,j}, \tau^{-1} Y_{i,j})$ form a basis for $\fh$. Note that $\tau^{-1} X_{i,j}=-X_{j,i}$ and $\tau^{-1} Y_{i,j}=\zeta Y_{j,i}$. A simple computation gives
\begin{align*}
X'_{i,j}(v_{k,\ell}) &= \delta_{j,k} v_{i,\ell} - \delta_{i,\ell} v_{k,j} &
Y'_{i,j}(v_{k,\ell}) &= -\zeta \cdot (\delta_{j,k} w_{i,\ell}+\delta_{i,\ell} w_{k,j}) \\
X'_{i,j}(w_{k,\ell}) &= \delta_{j,k} w_{i,\ell} - \delta_{i,\ell} w_{k,j} &
Y'_{i,j}(w_{k,\ell}) &= -\zeta \cdot (\delta_{j,k} v_{i,\ell}-\delta_{i,\ell} v_{k,j})
\end{align*}
Recall that the $x_{i,j}$ transform exactly like the $v_{i,j}$, and the $y_{i,j}$ transform exactly like the $w_{i,j}$. Note that
\begin{displaymath}
\delta_{j,k} x_{i,\ell}-\delta_{i,\ell} x_{k,j} = \delta_{j,k} (x_{i,\ell}-\delta_{i,\ell}) - \delta_{i,\ell} (x_{k,j}-\delta_{k,j}).
\end{displaymath}
We thus see that $X'_{i,j}$ and $Y'_{i,j}$ map each generator of $\fm$ to a linear combination of generators, which completes the proof.
\end{proof}

\begin{remark} \label{rmk:SpecA}
The space $\Spec(A)$ parametrizes bilinear forms $\langle, \rangle \colon \bV \otimes \bV \to \bC$ satisfying 
\begin{displaymath}
\langle \alpha(v), \alpha(w) \rangle = (-1)^{\vert v \vert} \zeta \langle v, w \rangle.
\end{displaymath}
The ideal $\fm$ corresponds to a standard form, and the above proposition simply says that the adjoint of an element $X$ of $\fq$ with respect to this form is given by $-\tau^{-1}X$.
\end{remark}

\begin{lemma} \label{lem:unit}
  If $\fa$ is a non-zero ideal of $A$, then $\fa+\fm=A$.
\end{lemma}

\begin{proof}
  Suppose $\fa$ is a non-zero ideal of $A$. Let $V_n$ and $W_n$ be copies of $\bC^{n|n}$ and let $A_n=\Sym(2^{-1}(V_n \otimes W_n))$, regarded as a subring of $\vert A \vert$. Then $A$ is the union of the $A_n$, and so for $n \gg 0$, $\fa'=\fa \cap A_n$ is a non-zero $\fq(n)\times \fq(n)$-stable ideal of $A_n$, and $\fm'=\fm \cap A_n$ is a maximal ideal of $A_n$. By Remark~\ref{rmk:SpecA}, $\Spec(A_n)$ is a space of certain bilinear forms on $V_n$, and $\fm' \in \Spec(A_n)$ has maximal rank. Since every ideal contains non-zero even degree elements, $V(\fa')$ is a proper closed $\fq(n) \times \fq(n)$-stable subset of $\Spec(A_n)$. Hence, it cannot contain any form of maximal rank (as the orbit of any such form is dense under the action of $\fgl(n) \times \fgl(n)= (\fq(n) \times \fq(n))_0$), and so $\fa' \not\subset \fm'$. It follows that $\fa \not\subset \fm$, and so $\fa+\fm=A$.
\end{proof}

\subsection{The group $K$}

Let $B$ be the supergroup consisting of all infinite matrices of the form
\begin{displaymath}
\begin{pmatrix} a & b \\ -b & a \end{pmatrix}
\end{displaymath}
with $a$ and $b$ upper triangular, $a$ even, $b$ odd, and $a$ invertible; we do \emph{not} require $a$ or $b$ to fix almost all basis vectors. This is naturally a group object in the category of super schemes. The coordinate ring $\bC[B]$ is the superalgebra $\bC[a_{i,j},b_{i,j},a_{i,i}^{-1}]_{j \ge i}$ where the $a$'s are even and the $b$'s are odd. If we multiply two matrices in $B$ then each entry of the result involves only finitely many entries of the original two matrices; the formulas for the entries of the result are used to define the comultiplication map $\bC[B] \to \bC[B] \otimes \bC[B]$. We let $U$ be the subgroup of $B$ where $a_{i,i}=1$ and $b_{i,i}=0$ for all $i$, and we let $K=B \times U$. We use coordinates $a_{i,j}$, $b_{i,j}$, $a'_{i,j}$, $b'_{i,j}$ on $K$, where the first pair are coordinates on $B$ the second on $U$. The Lie algebra of $B$ is essentially the algebra $\fb$ defined in the previous section, except its elements can have infinitely many non-zero entries. Similarly, the Lie algebra of $U$ is essentially $\fn$, and that of $K$ is essentially $\fk$.

If $V$ is a polynomial representation of $\fq$ then its restriction to $\fb$ naturally extends to a representation of $B$ (i.e., it admits a comodule structure over $\bC[B]$). Similarly, if $V$ is a polynomial representation of $\fq \times \fq$ then its restriction to $\fk$ naturally extends to a representation of $K$.

\subsection{The map $\phi$}

Let $M$ be an $A$-module. We then have the comultiplication map $M \to M \otimes \bC[K]$ as discussed above. Composing with the quotient map $M \to M/\fm M$, we obtain a linear map
\begin{displaymath}
\phi_M \colon M \to M/\fm M \otimes \bC[K].
\end{displaymath}
In fact, we can define $\phi_M$ for any $K$-equivariant $\vert A \vert$-module (where ``$K$-equivariant'' means the action of $K$ comes from a comodule structure). We now study this map, beginning with the case $M=A$:

\begin{proposition}
We have the following:
\begin{enumerate}
\item The map $\phi_A \colon A \to \bC[K]$ is the $\bC$-algebra homomorphism given by
\begin{displaymath}
\phi(x_{i,j}) = \sum_{k \le i,j} a_{k,i} a'_{k,j} + \zeta b_{k,i} b'_{k,j}, \qquad
\phi(y_{i,j}) = \sum_{k \le i,j} a_{k,i} b'_{k,j} - \zeta b_{k,i} a'_{k,j}
\end{displaymath}
\item The extension of $\fm$ along the map $\phi_A$ is the maximal ideal of $\bC[K]$ corresponding to the identity element of $K$.
\item The map $\phi_A$ induces an isomorphism of localizations $A_{\fm} \to \bC[K]_{\fm}$.
\end{enumerate}
\end{proposition}

\begin{proof}
(a) Let $g$ be an element of $K$ with coordinates $a_{i,j}$, $b_{i,j}$, $a'_{i,j}$, $b'_{i,j}$. We have
\begin{align*}
g (e_i \otimes e_j) &= \left( \sum_{k \le i} a_{k,i} e_k - b_{k,i} f_k \right) \otimes \left( \sum_{\ell \le j} a'_{\ell,j} e_\ell - b'_{\ell,j} f_\ell \right) \\
g (e_i \otimes f_j) &= \left( \sum_{k \le i} a_{k,i} e_k - b_{k,i} f_k \right) \otimes \left( \sum_{\ell \le j} b'_{\ell,j} e_\ell + a'_{\ell,j} f_\ell \right) \\
g (f_i \otimes e_j) &= \left( \sum_{k \le i} b_{k,i} e_k + a_{k,i} f_k \right) \otimes \left( \sum_{\ell \le j} a'_{\ell,j} e_\ell + b'_{\ell,j} f_\ell \right) \\
g (f_i \otimes f_j) &= \left( \sum_{k \le i} b_{k,i} e_k + a_{k,i} f_k \right) \otimes \left( \sum_{\ell \le j} -b'_{\ell,j} e_\ell + a'_{\ell,j} f_\ell \right) \\
\end{align*}
We thus find
\begin{align*}
gv_{i,j} &= \sum_{k \le i, \ell \le j} (a_{k,i} a'_{\ell,j}+\zeta b_{k,i} b'_{\ell,j}) v_{k,\ell}-(a_{k,i} b'_{\ell,j}+\zeta b_{k,i} a'_{\ell,j}) w_{k,\ell} \\
gw_{i,j} &= \sum_{k \le i, \ell \le j} (a_{k,i} b'_{\ell,j}-\zeta b_{k,i} a'_{\ell,j}) v_{k,\ell}+(a_{k,i} a'_{\ell,j}-\zeta b_{k,i} b'_{\ell,j}) w_{k,\ell}
\end{align*}
from which the stated formulas for $\phi(x_{i,j})$ and $\phi(y_{i,j})$ follow (change $v_{i,j}$ to $x_{i,j}$ and $w_{i,j}$ to $y_{i,j}$, and reduce modulo $\fm$, which takes $x_{i,j}$ to $\delta_{i,j}$ and $y_{i,j}$ to 0). Since $K$ acts on $A$ by algebra homomorphisms, it follows that $\phi$ is an algebra homomorphism.

(b) It suffices to show that the elements $a_{i,j} - \delta_{i,j}, b_{i,j}, a'_{i,j}$ and $b'_{i,j}$ are in the extension $\fm^e$ of $\fm$. We do this lexicographically on the pair $(i, j)$ (with $j$ being more significant). The base case $(1,1)$ follows from: \begin{align*}
\phi(x_{1,1} - 1) &= a_{1,1} - 1\\
\phi(y_{1,1}) & = -\zeta b_{1,1}.
\end{align*} Suppose the result hold for all pairs smaller than $(i,i)$. Then we have the following:
\begin{align*}
\phi(x_{i,i} - 1) &= a_{i,i} - 1 \pmod{\fm^e} \\
\phi(y_{i,i}) & = -\zeta b_{i,i} \pmod{\fm^e}, 
\end{align*}  from which it follows that $a_{i,i} - 1, b_{i,i} \in \fm^e$. Next suppose $i < j$ and assume, by induction, that the result holds for all pairs smaller than $(i,j)$. Then we have the following: 
\begin{align*}
\phi(x_{i,j}) &= a'_{i,j}  \pmod{\fm^e} \\
\phi(y_{i,j}) & = b'_{i,j} \pmod{\fm^e} \\
\phi(x_{j,i}) &= a_{i,j} \pmod{\fm^e}\\
\phi(y_{j,i}) &= -\zeta b_{i,j} \pmod{\fm^e}.
\end{align*} 
It follows that $a'_{i,j}, b'_{i,j}, a_{i,j}, b_{i,j} \in \fm^e$, which proves (b).

(c) We now define an algebra morphism $\psi \colon \bC[K] \to A_{\fm}$ such that the kernel of the composition $\bC[K] \to A_{\fm} \to A_{\fm}/\fm A_{\fm} $ is the ideal $\fm^e$. We define $\psi$ lexicographically on the generators as follows. In the base case, we set \begin{align*}
\psi(a_{1,1}) &= x_{1,1}\\
\psi(b_{1,1}) &= \zeta y_{1,1}. 
\end{align*} Suppose $i \le j$ and assume, by induction, that $\psi$ has been defined on all generators indexed smaller than $(i,j)$. We set \begin{align*}
\psi(a_{i,j}) &= x_{j,i} - \sum_{k<i} \psi( a_{k,j}a'_{k,i} +\zeta b_{k, j} b'_{k,i} )   \\
\psi(b_{i,j}) &= \zeta y_{j,i}  - \zeta \sum_{k<i} \psi(a_{k,j}b'_{k,i} - \zeta b_{k,j} a'_{k,i}).
\end{align*} Moreover, if $i < j$ we set  \begin{align*}
\vec{\psi(a'_{i,j})}{\psi(b'_{i,j})} = \mat{\psi(a_{i,i})}{\zeta \psi( b_{i,i})}{-\zeta \psi(b_{i,i})}{\psi(a_{i,i})}^{-1}\vec{x_{i,j} - \sum_{k <i}\psi(a_{k,i}a'_{k,j} + \zeta b_{k,i} b'_{k,j}  ) }{y_{i,j} - \sum_{k <i}\psi(a_{k,i}b'_{k,j} -\zeta b_{k,i}a'_{k,j} ) }.
\end{align*}  One can check lexicographically that the kernel of the composition $\bC[K] \to A_{\fm} \to A_{\fm}/\fm A_{\fm} $ contains $\fm^e$, and so it must equal $\fm^e$. In particular, the determinant of the matrix above is $1 \pmod{\fm A_{\fm} }$, and so $\psi$ is well-defined. Since $\psi^{-1}(\fm A_{\fm} ) = \fm^e$, it follows that $\psi$ extends to a homomorphism $\wt{\psi} \colon \bC[K]_{\fm^e} \to A_{\fm}$. Let  $ \wt{\phi} \colon  A_{\fm} \to  \bC[K]_{\fm^e}$ be the map induced by $\phi$.  By the construction of $\wt{\psi}$, it is clear that $\wt{\psi} \circ \wt{\phi}$ is the identity map. Moreover, arguing lexicographically, one can check that $\wt{\phi} \circ \wt{\psi}$ is also the identity map. Thus $\wt{\phi}$ is an isomorphism.
\end{proof}

In what follows, we regard $\bC[K]$ as an $A$-module via $\phi_A$. We now study the map $\phi_M$ for an arbitrary $A$-module $M$. Since $K$ acts on $M$ by $A$-semilinear automorphisms, it follows that $\phi_M$ is a morphism of $A$-modules. Our goal is to prove that $\phi_M$ induces an isomorphism after localizing at $\fm$, which will establish the theorem. We require some lemmas first.

\begin{lemma} \label{lem:phi-modm}
Let $M$ be a $K$-equivariant $\vert A \vert$-module. Then $\phi_M$ induces an isomorphism modulo $\fm$.
\end{lemma}

\begin{proof}
See \cite[Lemma~4.4]{symsp1}.
\end{proof}

\begin{lemma} \label{lem:phi-locm}
Let $M$ be an $A$-module. Then the localization of $\phi_M$ at $\fm$ is surjective.
\end{lemma}

\begin{proof}
See \cite[Lemma~4.5]{symsp1}.
\end{proof}

\begin{lemma} \label{lem:ker}
Let $M$ be an $A$-module. Then the kernel of $\phi_M$ is $\fq \times \fq$-stable, and thus an $A$-submodule of $M$.
\end{lemma}

\begin{proof}
This is proved exactly like \cite[Lemma~4.8]{symsp1}, using an obvious generalization of \cite[Lemma~4.7]{symsp1} to superalgebras.
\end{proof}

By an ``polynomially $\fh$-equivariant $A_{\fm}$-module,'' we mean an $A_{\fm}$-module $N$ equipped with a compatible action of $\fh$ such that for every $x \in N$ there is a unit $u \in A_{\fm}$ such that $ux$ generates a polynomial $\fh$-representation (recall that $\fh \cong \fq$). The localization of any $A$-module at $\fm$ is a polynomially $\fh$-equivariant $A_{\fm}$-module. Any submodule or quotient module of a polynomially $\fh$-equivariant $A_{\fm}$-module (in the category of equivariant modules) is again polynomially equivariant.

\begin{lemma} \label{lem:eqfg}
Suppose that
\begin{displaymath}
0 \to R \to M \to N \to 0
\end{displaymath}
is an exact sequence of polynomially $\fh$-equivariant $A_{\fm}$-modules such that $M$ is equivariantly finitely generated and $N$ is free as an $\vert A_{\fm} \vert$-module. Then $R$ is also equivariantly finitely generated.
\end{lemma}

\begin{proof}
The argument in \cite[Lemma~5.8]{periplectic} applies here.
\end{proof}

\begin{lemma} \label{lem:nak}
Let $M$ be an $A$-module such that $M=\fm M$. Then $M_{\fm}=0$.
\end{lemma}

\begin{proof}
In what follows, $\fq$ is identified with the subalgebra $\fh \subset \fq \times \fq$. It suffices to prove the lemma when $M$ is finitely generated, so we assume this in what follows.
  
  Let $V \subset M$ be a finite length $\GL$-subrepresentation generating $M$ as an $A$-module. Pick $m_1, \ldots, m_k \in V$ such that the $m_i$ generate $V(\bC^{N\mid N})$ as a $\fq(N)$-representation for all $N \gg 0$ ($V$ is a sum of finitely many irreducible representations, which are generated by their highest weight spaces, so it suffices to pick $N$ large enough so that each of these highest weight spaces are nonzero, and pick the $m_i$ to span the highest weight spaces in $V(\bC^{N\mid N})$). Write $m_i=\sum_i a_{i,j} n_{i,j}$ where $a_{i,j} \in \fm$ and $n_{i,j} \in M$.

  Let $N \gg 0$ be large enough so that the $m_i$ and the $n_{i,j}$ belong to $M'=M(\bC^{N\mid N})$ and the $a_{i,j}$ belong to $A'=A(\bC^{N \mid N})$. Let $V'=V(\bC^{N\mid N})$ and let $\fm'$ be the ideal of $A'$ generated by $x_{i,j}-\delta_{i,j}$ and $y_{i,j}$ for $i,j \le N$. Then $M'$ is an $A'$-module and generated (ignoring any Lie superalgebra action) by $V'$. We have $m_i \in \fm' M'$ for all $i$, and so $gm_i \in \fm' M'$ for any $g \in \fq(N)$, since $\fm'$ is $\fq(N)$-stable. Thus $V' \subset \fm' M'$ and so $M'=\fm' M'$. Thus, by the usual version of Nakayama's lemma \cite[(4.22)]{lam}, we have $M'_{\fm'}=0$. Therefore, for each $1 \le i \le k$ there exists homogeneous degree $0$ elements $s_i \in A'\setminus \fm' \subset A \setminus \fm$ such that $s_im_i=0$, which implies that $M_\fm=0$.
\end{proof}

We now reach the main result:

\begin{proposition}
Let $M$ be an $A$-module. Then $\phi_M$ induces an isomorphism after localizing at $\fm$.
\end{proposition}

\begin{proof}
The argument from \cite[Proposition~4.11]{symsp1} applies. For the convenience of the reader, we produce it here:

The assignment $M \mapsto \phi_M$ commutes with filtered colimits, and so it suffices to treat the case where $M$ is finitely generated. Let $R$ be the kernel of $\phi_M$, which is an $A$-submodule of $M$ by Lemma~\ref{lem:ker}, and let $N=M/\fm M \otimes \bC[K]$. By Lemma~\ref{lem:phi-locm}, the localization of $\phi_M$ at $\fm$ is a surjection. Since localization is exact, we have an exact sequence of polynomially $\fh$-equivariant $A_{\fm}$-modules
\begin{displaymath}
0 \to R_{\fm} \to M_{\fm} \to N_{\fm} \to 0.
\end{displaymath}
From Lemma~\ref{lem:eqfg}, we conclude that $R_{\fm}$ is equivariantly finitely generated as an $A_{\fm}$-module. Let $V \subset R$ be a finite length polynomial $\fq \times \fq$-representation generating $R_{\fm}$ as an $\vert A_{\fm} \vert$-module, and let $R_0$ be the $A$-submodule of $R$ generated by $V$. Note that $R_0$ is finitely generated as an $A$-module and $(R_0)_{\fm}=R_{\fm}$. Now, the mod $\fm$ reduction of the above exact sequence is exact, by the freeness of $N_{\fm}$, and the reduction of $M_{\fm} \to N_{\fm}$ is an isomorphism by Lemma~\ref{lem:phi-modm}. We conclude that $R/\fm R=R_0/\fm R_0=0$. Lemma~\ref{lem:nak} thus shows that $(R_0)_{\fm}=0$ and so $R_{\fm}=0$, and the proposition is proved.
\end{proof}

\section{The generic category} \label{s:gen}

\subsection{Algebraic representations of the infinite isomeric algebra} \label{ss:isomeric-alg}

Recall that $\bV=\bC^{\infty|\infty}$ is our fixed isomeric vector space with basis $\{e_i,f_i\}_{i \ge 1}$. The dual space $\bV^*$ is naturally a representation of $\fq$. We define the restricted dual, denoted $\bV_*$, to be the subspace of $\bV^*$ spanned by the dual basis vectors $e_i^*$ and $f_i^*$. This is a $\fq$-subrepresentation of $\bV^*$, and isomorphic to the direct limit of the spaces $(\bC^{n|n})^*$. We say that a representation of $\fq$ is {\bf algebraic} if it occurs as a subquotient of a (perhaps infinite) direct sum of mixed tensor spaces $T_{n,m}=\bV^{\otimes n} \otimes \bV_*^{\otimes m}$. We let $\Rep^{\alg}(\fq)$ be the category of algebraic representations, and $\Rep^{\alg,\rf}(\fq)$ the subcategory spanned by objects of finite length. This category is studied in detail in \cite{serganova}, where it is denoted $\operatorname{Trep}(\fg)$; see \cite[Definition~3.2]{serganova} for the precise definition. (Similar categories for classical Lie algebras were studied in \cite{koszulcategory, penkovserganova, penkovstyrkas, infrank}.) We will require some results from \cite{serganova}, which we now review.

\begin{proposition} \label{prop:qalg}
We have the following:
\begin{enumerate}
\item The representation $T_{n,m}$ has finite length.
\item Every algebraic representation of $\fq$ is the union of its finite length subobjects.
\item We have $\Rep^{\alg,\rf}(\fq) = \operatorname{Trep}(\fg)$.
\item For strict partitions $\lambda$ and $\mu$, there is a simple object $V_{\lambda,\mu}$ of $\Rep^{\alg}(\fq)$, and every simple object is isomorphic to $V_{\lambda,\mu}$ or $V_{\lambda,\mu}[1]$ for some $\lambda$ and $\mu$.
\item For strict partitions $\lambda$ and $\mu$, the representation $\bT_\lambda(\bV) \otimes \bT_\mu(\bV_*)$ is injective in $\Rep^{\alg}(\fq)$; in fact, it is the injective envelope of $V_{\lambda,\mu}$.
\item Every finite length object of $\Rep^{\alg}(\fq)$ has finite injective dimension.
\end{enumerate}
\end{proposition}

\begin{proof}
(a) is explained in the paragraph following \cite[Definition~3.2]{serganova}. (b) follows from (a) and the definition of algebraic representation.

We now explain (c). Since $T_{n,m}$ belongs to $\operatorname{Trep}(\fg)$ (see the paragraph following \cite[Definition~3.2]{serganova}) and $\operatorname{Trep}(\fg)$ is an abelian subcategory of the category of all $\fq$-modules, it follows that $\Rep^{\alg,\rf}(\fq) \subset \operatorname{Trep}(\fg)$. By \cite[Lemma~3.10]{serganova} and \cite[Corollary~4.3]{serganova}, every simple object of $\operatorname{Trep}(\fg)$ embeds into some $T_{n,m}$. By \cite[Proposition~4.10]{serganova}, $T_{n,m}$ is injective in $\operatorname{Trep}(\fg)$, and so we see that every object of $\operatorname{Trep}(\fg)$ embeds into a sum of $T_{n,m}$'s. This gives the reverse inclusion.

(d) follows from \cite[Lemma~3.10]{serganova}. (e) follows from \cite[Proposition 4.11]{serganova} (this only gives injectivity in $\Rep^{\alg,\rf}(\fq)$, but it is easy to see that the representation remains injective in $\Rep^{\alg}(\fq)$). (f) follows from \cite[Lemma~5.13]{serganova}.
\end{proof}

\subsection{Torsion modules}

We now define a notion of ``torsion'' for $A$-modules. We begin with the following observation:

\begin{proposition}
Let $M$ be an $A$-module. The following conditions are equivalent:
\begin{enumerate}
\item For every finitely generated submodule $M'$ of $M$ there is a non-zero ideal $\fa$ of $A$ such that $\fa M=0$. (Here ``ideal'' implies $\fq \times \fq$-stable.)
\item We have $M_{\fm}=0$.
\item For every $m \in M$ there exists $a \in A$ with non-zero image in $\bC[x_{i,j}]=A/(y_{i,j})$ such that $am=0$.
\end{enumerate}
\end{proposition}

\begin{proof}
First suppose that (a) holds. Let $M' \subset M$ be finitely generated with non-zero annihilator $\fa$. Since $\fa+\fm=A$ by Lemma~\ref{lem:unit}, we have $\fm M'=M'$, and so $M'_{\fm}=0$ by Lemma~\ref{lem:nak}. Since this holds for all finitely generated $M' \subset M$, it follows that $M_\fm=0$ and hence (b) holds.

Now suppose (b) holds. Given $m \in M$, there exists $s \in A \setminus\fm$ such that $sm=0$. As $s$ has non-zero reduction in $\bC[x_{i,j}]$, one can take $a=s$ in (c). Thus (c) holds.

Finally, suppose (c) holds. Let $M'$ be a submodule of $M$ generated by $m_1, \ldots, m_k$. Let $a_i m_i=0$ with $a_i$ as in (c). Let $a=a_1 \cdots a_k$, which has non-zero image in $A/(y_{i,j})$ since $\bC[x_{i,j}]$ is a domain, and annihilates each $m_i$. Following the proof of \cite[Prop.~2.2]{sym2noeth}, we see that there exists $n$, depending only on the $m_i$, such that $a^n(gm_i)=0$ for all $g \in \fq \times \fq$. Now (a) holds by taking $\fa$ to be the (non-zero) ideal of $A$ generated by $a^n$.
\end{proof}

We say that an $A$-module is {\bf torsion} if it satisfies the equivalent conditions of the above proposition. We write $\Mod_A^{\tors}$ for the category of torsion modules. It is clearly a Serre subcategory of $\Mod_A$.

\subsection{Statement of results}

We define the {\bf generic category} $\Mod_A^{\gen}$ to be the Serre quotient $\Mod_A/\Mod_A^{\tors}$. We write $T \colon \Mod_A \to \Mod_A^{\gen}$ for the localization functor and let $S$ be its right adjoint (the section functor). The goal of this section is to understand the structure of the generic category and the behavior of $T$ and $S$. We accomplish this by relating the generic category to $\Rep^{\alg}(\fq)$.

Let $M$ be an $A$-module. Define $\Phi(M)=M/\fm M$, where $\fm$ is the maximal ideal of $\vert A \vert$ considered in the previous section. Since $\fm$ is stable under $\fh \subset \fq \times \fq$, it follows that $\Phi(M)$ is naturally a representation of $\fh \cong \fq$. It is easily seen to be algebraic: indeed, expressing $M$ as a quotient of $A \otimes V$, with $V$ a polynomial representation of $\fq \times \fq$, we see that $\Phi(M)$ is a quotient of $V \vert_{\fh}$, and thus algebraic. We have thus defined a functor
\begin{displaymath}
\Phi \colon \Mod_A \to \Rep^{\alg}(\fq).
\end{displaymath}
Since $\Phi$ is cocontinuous and its source and target are Grothendieck abelian categories, it admits a right adjoint $\Psi$. The following is the main theorem of this section:

\begin{theorem} \label{thm:gen}
We have the following:
\begin{enumerate}
\item The functor $\Phi$ is exact.
\item The kernel of $\Phi$ is $\Mod_A^{\tors}$.
\item The counit $\Phi \Psi \to \id$ is an isomorphism.
\item The functor $\Phi$ induces an equivalence $\Mod_A^{\gen} \to \Rep^{\alg}(\fq)$.
\item The unit $V \otimes A \to \Psi(\Phi(V \otimes A))$ is an isomorphism for any $V \in \Rep^{\pol}(\fq \times \fq)$.
\end{enumerate}
\end{theorem}

Appealing to our knowledge of $\Rep^{\alg}(\fq)$ (Proposition~\ref{prop:qalg}), we find:

\begin{corollary} \label{cor:gen}
We have the following:
\begin{enumerate}
\item If $M$ is a finitely generated $A$-module then $T(A)$ has finite length.
\item Every finite length object of $\Mod_A^{\gen}$ has finite injection dimension.
\item The injectives of $\Mod_A^{\gen}$ are exactly the objects $T(V \otimes A)$ with $V \in \Rep^{\pol}(\fq \times \fq)$.
\item The unit $V \otimes A \to S(T(V \otimes A))$ is an isomorphism, for any $V \in \Rep^{\pol}(\fq \times \fq)$.
\end{enumerate}
\end{corollary}

Theorem~\ref{thm:gen} is analogous to \cite[Theorem~5.1]{symsp1}, \cite[Theorem~3.1]{sym2noeth}, \cite[Theorem~6.1]{periplectic}. We follow the proof of \cite{symsp1}, which is modeled on the proofs in the other two papers but contains some simplifications.

\subsection{An equality of dimensions} \label{ss:equality}

We will require the following result in our proof of Theorem~\ref{thm:gen}:

\begin{proposition} \label{prop:dim}
  For strict partitions $\lambda,\mu,\alpha,\beta$, we have
  \begin{align*}
    \dim \Hom_{\fq(\bV) \times \fq(\bW)}(\bT_\lambda(\bV) \otimes \bT_\mu(\bW), \bT_\alpha(\bV) \otimes \bT_\beta(\bW) \otimes A)\\
    = \dim \Hom_{\fq(\bV)}(\bT_\lambda(\bV) \otimes \bT_\mu(\bV_*), \bT_\alpha(\bV) \otimes \bT_\beta(\bV_*))
  \end{align*}
\end{proposition}

\begin{proof}
Let
  \[
    f^\lambda_{\mu,\nu} = \dim \Hom_{\fq(\bV)}(\bT_\lambda(\bV), \bT_\mu(\bV) \otimes \bT_\nu(\bV)).
  \]
  By \eqref{eqn:Q-cauchy}, we have
\[
  \Sym(2^{-1}(\bV \otimes \bW)) = \bigoplus_\gamma 2^{-\delta(\gamma)} \bT_\gamma(\bV) \otimes \bT_\gamma(\bW).
\]

We now have two cases. If $|\lambda|-|\alpha|\ne |\mu|-|\beta|$, then the left side is $0$ (by \cite[Corollary 3.7]{serganova}, if $f^\lambda_{\mu,\nu} \ne 0$, then $|\lambda|=|\mu|+|\nu|$). The right side is also 0 by \cite[Corollary 5.11]{serganova}. Otherwise, suppose that $r=|\lambda|-|\alpha|=|\mu|-|\beta|$. Then the left side simplifies to
  \[
    \sum_{|\gamma|=r} 2^{-\delta(\gamma)} f^\lambda_{\alpha, \gamma} f^\mu_{\beta, \gamma}.
  \]
  This is also the right hand side by \cite[Theorem 5.8]{serganova} ($Z(\lambda,\mu)$ is defined to be  $2^{-\delta(\lambda)}(\bT_\lambda(\bV) \otimes \bT_\lambda(\bV^*))$ in \cite[\S 4.1]{serganova}, which accounts for the difference in the form of the formula).
\end{proof}

\subsection{Proof of Theorem~\ref{thm:gen}}

The proofs of the following lemmas are nearly identical to those in \cite[\S 5.2]{symsp1}, so we omit most details. For strict partitions $\lambda$ and $\mu$, let $F_{\lambda,\mu}=\bT_{\lambda}(\bV) \otimes \bT_{\mu}(\bV) \otimes A$, and let $\cF$ be the class of $A$-modules that are direct sums (perhaps infinite) of $F_{\lambda,\mu}$'s.

\begin{lemma}
Let $f \colon M \to N$ be a morphism of $A$-modules such that $\Phi(f)=0$. Then the localized morphism $f_{\fm} \colon M_{\fm} \to N_{\fm}$ vanishes.
\end{lemma}

\begin{proof}
See \cite[Lemma~5.3]{symsp1}.
\end{proof}

\begin{lemma}
For any $F \in \cF$ and any strict partitions $\lambda$ and $\mu$, the map
\begin{displaymath}
\Phi \colon \Hom_A(F_{\lambda,\mu}, F) \to \Hom_{\fq}(\Phi(F_{\lambda,\mu}), \Phi(F))
\end{displaymath}
is an isomorphism.
\end{lemma}

\begin{proof}
See \cite[Lemma~5.4]{symsp1}, and use Proposition~\ref{prop:dim} in place of \cite[Proposition~3.9]{symsp1}.
\end{proof}

\begin{lemma}
Let $f \colon M \to N$ be a morphism of $A$-modules. Suppose that for all strict partitions $\lambda$ and $\mu$ the induced map
\begin{displaymath}
f_* \colon \Hom_A(F_{\lambda,\mu}, M) \to \Hom_A(F_{\lambda,\mu}, N)
\end{displaymath}
is an isomorphism. Then $f$ is an isomorphism.
\end{lemma}

\begin{proof}
This follows from the fact that the $F_{\lambda,\mu}$ generate $\Mod_A$, see \cite[Lemma~5.6]{symsp1}.
\end{proof}

\begin{lemma}
For $F \in \cF$, the unit $\eta_F \colon F \to \Psi(\Phi(F))$ is an isomorphism.
\end{lemma}

\begin{proof}
See \cite[Lemma~5.7]{symsp1}.
\end{proof}

\begin{lemma}
Let $I$ be an injective object of $\Rep^{\alg}(\fq)$. Then the counit $\epsilon_I \colon \Phi(\Psi(I)) \to I$ is an isomorphism.
\end{lemma}

\begin{proof}
See \cite[Lemma~5.8]{symsp1}.
\end{proof}

The proof of Theorem~\ref{thm:gen} now follows in the same way as the proof of \cite[Theorem~5.1]{symsp1}.

\section{Proof of the noetherianity theorem} \label{s:noeth}

We can now finally prove the noetherianity theorem:

\begin{theorem} \label{thm:noeth}
The bivariate isomeric algebra $A$ is noetherian.
\end{theorem}

The proof will occupy the remainder of this section. We say that an $A$-module $M$ satisfies (FT) if $\Tor_i^A(M, \bC)$ is a finite length $\fq \times \fq$-module for all $i \ge 0$. This implies $M$ is finitely generated (it suffices to just know that $\Tor_0^A(M,\bC)$ is finite length) by Nakayama's lemma.

\begin{lemma} \label{lem:FT}
Let $M$ be a finitely generated torsion $A$-module. Then $M$ satisfies (FT) and $M$ is noetherian.
\end{lemma}

\begin{proof}
  Now let $M$ be a finitely generated torsion $A$-module. Then there is a non-zero ideal $I \subset A$ which annihilates $M$. By Corollary~\ref{cor:A-bounded}, $A/I$ is bounded, so $M$ is noetherian by Proposition~\ref{prop:bounded-noeth}. Next, $\Tor_i^A(M, \bC)$ is computed by the Koszul complex, and hence is a subquotient of $\bigwedge^i(\bU) \otimes M$, which is bounded. Hence it can be computed by specializing to finitely many variables, which implies that it has finite length.
\end{proof}

\begin{lemma} \label{lem:FT2}
If $M$ is a finite length object of $\Mod_A^{\gen}$ then $S(M)$ satisfies (FT).
\end{lemma}

\begin{proof}
The proof is essentially the same as \cite[Proposition~4.8]{sym2noeth}. We recall the details. We in fact show that $(\rR^i S)(M)$ satisfies (FT) for all $i \ge 0$. We proceed by induction on the injective dimension of $M$, which is finite by Corollary~\ref{cor:gen}(b).

First suppose that $M$ is injective. Then $M=T(V \otimes A)$ for some finite length $V \in \Rep^{\pol}(\fq \times \fq)$ (Corollary~\ref{cor:gen}), and so $S(M)=V \otimes A$ (Corollary~\ref{cor:gen}), which obviously satisfies (FT); of course, $(\rR^i S)(M)=0$ for $i>0$ since $M$ is injective. Thus the claim holds.

Now suppose that $M$ has positive injective dimension. Choose a short exact sequence
\begin{displaymath}
0 \to M \to I \to N \to 0
\end{displaymath}
where $I$ is injective and $N$ has smaller injective dimension than $M$. Then $(\rR^i S)(M)=(\rR^{i-1} S)(N)$ for $i \ge 2$, and thus satisfies (FT) by the inductive hypothesis. We have an exact sequence
\begin{displaymath}
0 \to S(M) \to S(I) \to S(N) \to (\rR^1 S)(M) \to 0.
\end{displaymath}
Since $S(N)$ satisfies (FT) by the inductive hypothesis, it is finitely generated, and so $(\rR^1 S)(M)$ is finitely generated. Since $(\rR^1 S)(M)$ is also torsion, it satisfies (FT) by Lemma~\ref{lem:FT}. It thus follows that $S(M)$ satisfies (FT), as the other three terms in the exact sequence do.
\end{proof}

\begin{proof}[Proof of Theorem~\ref{thm:noeth}]
The proof is essentially the same as \cite[Theorem~4.9]{sym2noeth}. We recall the details. Let $P$ be a finitely generated projective $A$-module and let $N_1 \subset N_2 \subset \cdots$ be an ascending chain of $A$-submodules. Then $T(N_{\bullet})$ is an ascending chain in $T(P)$, and thus stabilizes, as $T(P)$ has finite length. Discarding finitely many terms, we assume that $T(N_{\bullet})$ is constant. Let $N'$ be the common value of $S(T(N_i))$. We have the following:
\begin{itemize}
\item $N'$ is finitely generated by Lemma~\ref{lem:FT2};
\item $N'$ is a submodule of $P=S(T(P))$ that contains each $N_i$; and
\item $N'/N_1$ is torsion, since $T(N'/N_1)=0$.
\end{itemize}
Since $N'/N_1$ is finitely generated and torsion, it is noetherian (Lemma~\ref{lem:FT}), so the ascending chain $N_{\bullet}/N_1$ in it stabilizes. It follows that $N_{\bullet}$ stabilizes, and so $P$ is noetherian.
\end{proof}


\begin{thebibliography}{NSS2}

\bibitem[Ab]{abeasis} Silvana Abeasis. The ${\rm GL}(V)$-invariant ideals in $S(S^{2}V)$. {\it Rend. Mat. (6)} {\bf 13} (1980), no.~2, 235--262.

\bibitem[AdF]{pfaffians} S.~Abeasis, A.~Del Fra. Young diagrams and ideals of Pfaffians. {\it Adv. in Math.} {\bf 35} (1980), no.~2, 158--178.

\bibitem[BE]{brundan} Jonathan Brundan, Alexander P. Ellis. Monoidal supercategories. {\it Comm. Math. Phys.} {\bf 351} (2017), no.~3, 1045--1089. \arxiv{1603.05928v3} 

\bibitem[CEP]{CEP} C.~de Concini, David Eisenbud, C.~Procesi. Young diagrams and determinantal varieties. {\it Invent. Math.} {\bf 56} (1980), no.~2, 129--165. 

\bibitem[CW1]{chengwang} Shun-Jen Cheng, Weiqiang Wang. {\it Dualities and Representations of Lie Superalgebras}. Graduate Studies in Mathematics {\bf 144}, American Mathematical Society, Providence, RI, 2012.

\bibitem[CW2]{chengwang2} Shun-Jen Cheng, Weiqiang Wang. Remarks on the Schur--Howe--Sergeev Duality. {\it Lett. Math. Phys.}~{\bf 52} (2000), 143--153. \arxiv{math/0008109v1}

\bibitem[DPS]{koszulcategory} Elizabeth Dan-Cohen, Ivan Penkov, Vera Serganova. A Koszul category of representations of finitary Lie algebras. \textit{Adv.\ Math.} {\bf 289} (2016), 250--278. \arxiv{1105.3407v2}.

\bibitem[GS]{serganova} Dimitar Grantcharov, Vera Serganova. Tensor representations of $\fq(\infty)$. {\it Algebr.\ Represent.\ Theory}~{\bf 22} (2019), 867--893. \arxiv{1605.02389v1}

\bibitem[Lam]{lam} T.~Y. Lam. {\it A first course in noncommutative rings}. Second edition, Graduate Texts in Mathematics~{\bf 131}, Springer-Verlag, New York, 2001.

\bibitem[Lan]{Lang} S.~Lang. \emph{Algebra}. Graduate Texts in Mathematics, Springer-Verlag, New York {\bf 211}, 2002.

\bibitem[LS]{LS} Robert Laudone, Andrew Snowden. Equivariant prime ideals for infinite dimensional supergroups. \arxiv{2103.03152v1}

\bibitem[Ma]{macdonald} I. G. Macdonald. {\it Symmetric Functions and Hall Polynomials}. Second edition, Oxford Mathematical Monographs, Oxford, 1995.

\bibitem[NSS1]{sym2noeth} Rohit Nagpal, Steven V Sam, Andrew Snowden. Noetherianity of some degree two twisted commutative algebras. {\it Selecta Math. (N.S.)} {\bf 22} (2016), no.~2, 913--937. \arxiv{1501.06925v2}

\bibitem[NSS2]{periplectic} Rohit Nagpal, Steven V Sam, Andrew Snowden. Noetherianity of some degree two twisted skew-commutative algebras. {\it Selecta Math. (N.S.)} {\bf 25} (2019), no.~1. \arxiv{1610.01078v2}

\bibitem[PSe]{penkovserganova} Ivan Penkov, Vera Serganova. Categories of integrable $sl(\infty)$-, $o(\infty)$-, $sp(\infty)$-modules. In {\it Representation Theory and Mathematical Physics}, Contemp. Math. {\bf 557}, AMS 2011, pp. 335--357. \arxiv{1006.2749v1}

\bibitem[PSt]{penkovstyrkas} Ivan Penkov, Konstantin Styrkas. Tensor representations of classical locally finite Lie algebras. In {\it Developments and trends in infinite-dimensional Lie theory}, Progr. Math. {\bf 288}, Birkh\"auser Boston, Inc., Boston, MA, 2011, pp. 127--150. \arxiv{0709.1525v1}

\bibitem[Sn]{tcaspec} Andrew Snowden. The spectrum of a twisted commutative algebra. \arxiv{2002.01152v1}

\bibitem[SS1]{infrank} Steven V Sam, Andrew Snowden. Stability patterns in representation theory. {\it Forum Math., Sigma}~{\bf 3} (2015), e11, 108 pp. \arxiv{1302.5859v2}

\bibitem[SS2]{symsp1} Steven~V Sam, Andrew Snowden. $\Sp$-equivariant modules over polynomial rings in infinitely many variables. {\it Trans. Amer. Math. Soc.}, to appear. \arxiv{2002.03243v2}

\bibitem[SS3]{brauercat1} Steven~V Sam, Andrew Snowden. The representation theory of Brauer categories I: triangular categories. \arxiv{2006.04328v1}

\bibitem[SS4]{supermon} Steven~V Sam, Andrew Snowden. Supersymmetric monoidal categories. \arxiv{2011.12501v1}

\end{thebibliography}
\end{document}